\documentclass[11pt, twoside]{article}

\usepackage[english]{babel}

\usepackage{bbm}
\usepackage{amssymb}
\usepackage{amsfonts}
\usepackage{amsmath}
\usepackage{amsthm}
\usepackage{color}
\usepackage{mathrsfs}
\usepackage{txfonts}
\usepackage{bbm}
\usepackage{enumerate}
\usepackage{anysize}
\usepackage{indentfirst}

\usepackage{latexsym}

\usepackage[colorlinks=true,
linkcolor=blue,
citecolor=red,
urlcolor=magenta,
backref=page
]{hyperref}

\allowdisplaybreaks

\newtheorem{theorem}{Theorem}[section]
\newtheorem{lemma}[theorem]{Lemma}
\newtheorem{corollary}[theorem]{Corollary}
\newtheorem{proposition}[theorem]{Proposition}
\newtheorem{example}[theorem]{Example}
\theoremstyle{definition}
\newtheorem{remark}[theorem]{Remark}
\newtheorem{definition}[theorem]{Definition}

\newcounter{assum}

\renewcommand{\appendix}{\par
\setcounter{section}{0}%
\setcounter{subsection}{0}%
\setcounter{subsubsection}{0}%
\gdef\thesection{\@Alph\c@section}%
\gdef\thesubsection{\@Alph\c@section.\@arabic\c@subsection}%
\gdef\theHsection{\@Alph\c@section.}%
\gdef\theHsubsection{\@Alph\c@section.\@arabic\c@subsection}%
\csname appendixmore\endcsname
}

\numberwithin{equation}{section}

\def\XXint#1#2#3{{\setbox0=\hbox{$#1{#2#3}{\int}$ }
\vcenter{\hbox{$#2#3$ }}\kern-.6\wd0}}

\def\XXint#1#2#3{{\setbox0=\hbox{$#1{#2#3}{\int}$ }
\vcenter{\hbox{$#2#3$ }}\kern-.55\wd0}}

\begin{document}

\arraycolsep=1pt

\title{\bf\Large Hardy--Littlewood Maximal Operator
and Two-Layer Muckenhoupt Weights
on Infinite Rooted $k$-Ary Trees
\footnotetext{ \hspace{-0.35cm} 2020 \emph{Mathematics Subject Classification}.
Primary 42B25; Secondary 05C05, 46E30, 47A30, 46E36.
\endgraf \emph{Key words and phrases}.  infinite rooted $k$-ary tree,
Hardy--Littlewood maximal operator,
two-layer Muckenhoupt class, weighted norm inequality.
\endgraf
This project is partially supported by the National Natural
Science Foundation of China (Grant
Nos. 12371093 and 12431006), the Beijing Natural
Science Foundation (Grant No. 1262011), and
the Fundamental Research Funds for the Central Universities
(Grant No. 2253200028).}}
\author{Dachun Yang, Wen Yuan and Mingdong Zhang}
\date{\today}
\maketitle

\vspace{-0.6cm}

\begin{center}
\begin{minipage}{13cm}
{\small {\bf Abstract}\quad
Let $k\geq 2$ be an integer,  $T$  an infinite rooted $k$-ary tree,
and  $M$ the Hardy--Littlewood maximal operator on $T$.
For any $p\in(0,\infty)$, we characterize the weight
$w$ such that $M$ is bounded on $L^p(w)$. To this end, we introduce
a two-layer Muckenhoupt weight class $\mathscr A_p$ and
prove that, for any $p\in(\frac{1}{2},\infty)$,
the boundedness of $M$ on $L^p(w)$, $w\in\mathscr A_p$,
and the exponential decay boundedness of
spherical averaging operators on $L^p(w)$ are mutually equivalent,
and that, when $p\in(0,\frac{1}{2}]$, there exists no weight $w$
such that $M$ is bounded on $L^p(w)$.
Moreover, for any $p\in(\frac{1}{2},\infty)$,
we establish the quantitative estimate, with the optimal exponent
$\frac{1}{p}$ of the weight constant, for
the boundedness of $M$ on $L^p(w)$.
For any $p\in(1,\infty)$, we also obtain two further
equivalent characterizations of the boundedness of $M$ on
$L^p(w)$, respectively, in terms of a global Sawyer-type
testing condition and an estimate for
the weighted product measure of distance incidence sets.
As applications, for any $p\in(\frac{1}{2},\infty)$,
under the assumption that $M$ is bounded on $L^p(w)$,
we establish the boundedness of exponentially decaying
kernel operators on $L^p(w)$ and weighted Fefferman--Stein
vector-valued inequalities.}
\end{minipage}
\end{center}

%\vspace{0.2cm}

\tableofcontents

%\vspace{0.2cm}

\section{Introduction}\label{sec:intro}

The Hardy--Littlewood maximal operator $M_{\mathbb{R}^n}$
on $\mathbb{R}^n$ is one of the fundamental objects in harmonic analysis and
plays an important role in related fields
(see, for instance, \cite[Chapter 2]{duo01} and \cite[Chapter II]{ste93}).
It is well known that $M_{\mathbb R^n}$ is of strong type
$(p, p)$ for any $p\in(1,\infty)$
and is of weak type $(1, 1)$. In \cite{muckenhoupt72},
Muckenhoupt introduced the classical $A_p(\mathbb R^n)$ weights
on $\mathbb{R}^n$ and proved that, for any $p\in(1,\infty)$,
$w\in A_p(\mathbb R^n)$ if and only if
$M_{\mathbb{R}^n}$ is bounded on the weighted Lebesgue space
$L^p(w)$ on $\mathbb{R}^n$.
Since this seminal work of Muckenhoupt, maximal operators and their
(weighted) inequalities have been studied in a wide variety of
metric measure spaces (see, for example,
\cite{am84,kairema13} for spaces of homogeneous type,
\cite{gm01,kosz18a,kosz18b,op02} for non-doubling spaces,
and \cite{bm12,st16,st19} for discrete graph settings).

It is well known that trees serve
as a fundamental and prototypical class of graphs.
Recall that a graph is called a \emph{tree}
if it is connected and contains no cycles.
Suppose that $k\geq 2$ is an integer.
A tree $T$ is called an \emph{infinite rooted $k$-ary tree}
if $T$ has a root vertex and every vertex of $T$ has exactly $k$ children.
Let $T$ be the infinite rooted $k$-ary tree, and let
$M$ be the Hardy--Littlewood maximal operator on $T$ in \eqref{eq-def-M}.
The main purpose of this article is to characterize the weight $w$
such that $M$ is bounded on the weighted Lebesgue space $L^p(w)$
on $T$ in \eqref{eq-def-weightedLp} for any given $p\in(0,\infty)$.

We first present a brief history on the study of $M$ on trees.
Recall that the \emph{degree} of a vertex of a tree is the number of
its neighbors. A tree is said to be \emph{homogeneous} if
every vertex has the same degree.
The study of the boundedness of $M$ on trees can be traced back to
the work of Rochberg and Taibleson \cite{rt91},
in which they established the weak type $(1,1)$
estimate for the Green operator associated with random walks on
trees of bounded degree. When specialized to the isotropic
random walk on homogeneous trees, their estimate implies that
$M$ is of weak type $(1,1)$.
In 2010, Cowling et al. \cite[Corollary 3.2]{cms10} gave a direct and geometric proof
of the weak type $(1,1)$ estimate for $M$ on homogeneous trees.
Independently, by means of a combinatorial expansion argument,
Naor and Tao \cite[Theorem 1.5]{nt10} proved that
$M$ is of weak type $(1,1)$ on infinite rooted $k$-ary trees, which, together with
the trivial $L^\infty$ estimate and the Marcinkiewicz interpolation theorem,
yields that $M$ is of strong type $(p,p)$ for any $p\in(1,\infty)$.
The combinatorial ideas of Naor and Tao \cite{nt10}
were subsequently adapted to the study of maximal inequalities
in several related settings; see, for example, \cite{st19} for infinite graphs,
\cite{ao25,cms26} for hyperbolic spaces, \cite{gr23,ls24} for fractional maximal
operators, and \cite{gr23,ors21,or22} for weighted variants.
For more developments concerning the boundedness of different maximal operators on
trees of bounded degree, we refer to \cite{lmsv25,ms24}.
We also refer to \cite{lmstv25,lstv23a,lstv23b,mstv25,osv25}
for various developments in harmonic
analysis on trees equipped with flow measures, including
the Calder\'on--Zygmund theory, Hardy and ${\rm BMO}$ spaces, and
the boundedness results for Riesz transforms and spectral multipliers
associated with the flow Laplacian.

We now turn to weighted estimates for $M$ on the infinite rooted $k$-ary tree
$T$, where $k\geq 2$ is an integer. To this end,
we first present some notation and basic properties of $T$.
In what follows, if no confusion can arise, we always
identify the tree $T$ with its \emph{vertex set}.
Let $d$ be the standard tree distance on $T$, i.e.,
for any $x,y\in T$, $d(x,y)$ is the number of edges
of the unique path joining $x$ and $y$. For any subset $A\subset T$,
let the \emph{notation $|A|$} denote the counting measure of $A$.
In this article, we always consider the
metric measure space $(T,d,|\cdot|)$.
Let $\mathbb{Z}_+:=\mathbb{N}\cup\{0\}$,
where $\mathbb{N}$ denotes the set of all positive integers.
For any $x\in T$ and $r\in\mathbb{Z}_+$, let
\begin{align*}
B(x,r):=\left\{y\in T:\,d(x,y)\leq r\right\}
\end{align*}
and
\begin{align*}
S(x,r):=\left\{y\in T:\,d(x,y)=r\right\}
\end{align*}
be respectively the \emph{ball} and the \emph{sphere} with center $x$ and radius $r$.
It is well known that, for any $x\in T$ and $r\in\mathbb{Z}_+$,
\begin{align}\label{eq-sphere-ball-growth}
|S(x,r)|\sim k^r\sim|B(x,r)|,
\end{align}
where the positive equivalence constants are independent of $k$, $x$, and $r$.
Note that the exponential growth in \eqref{eq-sphere-ball-growth}
shows that the counting measure on $T$ is not doubling
and even not upper doubling introduced
in \cite[Definition 2.6]{hytonen10}.
This prevents a direct application of the methods developed
in the study of weighted norm inequalities
on graphs as in \cite{bm12,bd16}, where the underlying
measure is assumed to be doubling.
The \emph{(centered) Hardy--Littlewood maximal operator} $M$ on $T$
is defined by setting, for any function $f$ on $T$
and for any $x\in T$,
\begin{align}\label{eq-def-M}
Mf(x):=\sup_{r\in\mathbb{Z}_+}A_rf(x)
:=\sup_{r\in\mathbb{Z}_+}\frac{1}{|B(x,r)|}
\sum_{y\in B(x,r)}|f(y)|.
\end{align}
Unlike in the Euclidean setting,
we do not need to assume that $f$ is locally integrable
in \eqref{eq-def-M} because
every ball in $T$ contains only finitely
many points and hence every function on $T$ is integrable on each ball.
We say that $w$ is a \emph{weight} on $T$ if $w$ is a function on $T$ taking values
in $(0,\infty)$. Recall that, for any $p\in(0,\infty)$
and any weight $w$ on $T$,
the \emph{weighted Lebesgue space $L^p(w)$ on $T$}
is defined to be the set of all functions $f$ on $T$ such that
\begin{align}\label{eq-def-weightedLp}
\|f\|_{L^p(w)}:=\left[\sum_{x\in T}
|f(x)|^p w(x)\right]^{\frac{1}{p}}<\infty.
\end{align}
If $w\equiv1$, we simply denote $L^p(w)$ by $L^p(T)$.
Recently, Ombrosi and Rivera-R\'{\i}os \cite{or22}
gave some sufficient conditions for the boundedness
of $M$ on $L^p(w)$ when $p\in(1,\infty)$. Moreover,
they found that, different from the Euclidean case,
the classical ball $A_p$ condition cannot characterize this boundedness.
To be precise, Ombrosi and Rivera-R\'{\i}os in
\cite[Example 2 in Theorem 1.3]{or22} constructed a weight
$w$ on $T$ such that $M$ is bounded on $L^q(w)$ for any
$q\in(1,\infty)$, but $w$ does not satisfy the classical ball
$A_p$ condition for any $p\in (1,\infty)$.
Thus, it is a \emph{natural question} to find those weights
$w$ on $T$ such that they can exactly
characterize the boundedness of $M$ on $L^p(w)$.

The main target of this article is to give an affirmative answer
to this question. To this end, we introduce the two-layer Muckenhoupt weights.
Different from the classical ball $A_p$ condition
in $\mathbb{R}^n$, this new condition is formulated in terms of
two descendant layers rather than a single ball and
contains an additional exponential decay factor.
To present its definition, we need the following notation.
In what follows, we use the symbol $o$ to denote the \emph{root} of $T$.
For any $x\in T$ and $m\in\mathbb Z_+$,
let $|x|:=d(o,x)$ and define
the \emph{$m$-th descendant layer $D_m(x)$
of $x$}  by setting
\begin{align*}
D_m(x):=\left\{y\in T:\ x\text{ lies on the unique path connecting }o\text{ and }y,
\ |y|=|x|+m\right\}.
\end{align*}
Since every vertex of $T$ has exactly $k$ children,
it follows that $|D_m(x)|=k^m$.
For any weight $w$ on $T$ and any subset $E\subset T$, let
$$
w(E):=\sum_{x\in E}w(x).
$$

We now present the definition of the two-layer
Muckenhoupt weights. In what follows,
for any $\alpha\in (0,\infty)$ and $y\in T$
we always let $w^{-\alpha}(y):=\frac{1}{w^\alpha(y)}$.

\begin{definition}\label{def:exponential-Ap}
For any $p,\varepsilon\in(0,\infty)$,
a weight $w$ is called a \emph{two-layer Muckenhoupt weight}, denoted by $w\in\mathscr{A}_{p,\varepsilon}$, if
\begin{align}\label{eq-two-layer-characterization}
[w]_{\mathscr{A}_{p,\varepsilon}}:=
\begin{cases}
\displaystyle
\sup_{\genfrac{}{}{0pt}{}{x\in T}{j,m\in\mathbb Z_+}}
\frac{w(D_j(x))[w^{-\frac{1}{p-1}}(D_m(x))
]^{p-1}}{k^{(p-\varepsilon)(j+m)}}
&\text{if }p\in(1,\infty),\\
\displaystyle
\sup_{\genfrac{}{}{0pt}{}{x\in T}{j,m\in\mathbb Z_+}}
\frac{\displaystyle w(D_j(x))\max_{y\in D_m(x)}w^{-1}(y)}{k^{(p-\varepsilon)(j+m)}}
&\text{if }p\in(0,1]
\end{cases}
\end{align}
is finite. We call
\begin{align*}
\mathscr{A}_p:=
\bigcup_{\varepsilon\in(0,\infty)}\mathscr{A}_{p,\varepsilon}
\end{align*}
the \emph{two-layer Muckenhoupt class
of index $p$} on $T$.
\end{definition}

\begin{remark}
\begin{itemize}
\item[{\rm (i)}] The following observations may illustrate
some reasonability of Definition \ref{def:exponential-Ap}.
Let $p\in(1,\infty)$.
\begin{itemize}
\item[$\mathrm{(i)}_{\mathrm{a}}$]
Recall that a non-negative locally integrable
function on $\mathbb{R}^n$ is called a \emph{weight} if it
takes values in $(0, \infty)$ almost everywhere
(see, for example, \cite[p.\,499]{g14c}).
For $x\in\mathbb R^n$ and $r\in(0,\infty)$, we also use
$B(x,r)$ to denote the Euclidean ball centered at $x$ with
radius $r$ and $|B(x,r)|$ to denote its Lebesgue measure.
If $w$ is a weight on $T$ or on $\mathbb{R}^n$,
let $\sigma:=w^{-\frac{1}{p-1}}$.
Recall that a weight $w$ on $\mathbb R^n$ is said to satisfy
the \emph{ball $A_p$ condition} if,
for any $x\in\mathbb R^n$ and $r\in(0,\infty)$,
\begin{align}\label{eq-Euclidean-one-ball}
\frac{1}{|B(x,r)|}\int_{B(x,r)}w(y)\,dy
\left[
\frac{1}{|B(x,r)|}
\int_{B(x,r)}\sigma(y)\,dy
\right]^{p-1}
\lesssim 1,
\end{align}
where the implicit positive constant is independent of $x$ and $r$
(see, for instance, \cite[Definition 7.1.3 and Remark 7.1.4]{g14c}).
Note that an analogue of \eqref{eq-Euclidean-one-ball} on $T$
seems the following one: for any $x\in T$ and $j\in\mathbb{Z}_+$,
\begin{align}\label{eq-single-layer-Ap}
\frac1{|D_j(x)|}\sum_{y\in D_j(x)}w(y)
\left[\frac1{|D_j(x)|}\sum_{y\in D_j(x)}\sigma(y)
\right]^{p-1}\lesssim 1,
\end{align}
where the implicit positive constant is independent of $x$ and $j$.
However, condition \eqref{eq-single-layer-Ap}
cannot be used to characterize the boundedness of $M$ on $L^p(w)$.
Indeed, for the exponential radial weights $w_\beta$ in
\eqref{eq-def-radial-beta}, the left-hand side of
\eqref{eq-single-layer-Ap} is equal to $1$ for all
$\beta\in\mathbb R$. But, by
Corollary \ref{cor:exponential-radial}, $M$ is bounded on
$L^p(w_\beta)$ if and only if $\beta\in (-1, p-1)$.
Hence condition \eqref{eq-single-layer-Ap} is
\emph{insufficient} for the boundedness of $M$ on $L^p(w)$.

\item[$\mathrm{(i)}_{\mathrm{b}}$]
On the other hand, observe that condition \eqref{eq-Euclidean-one-ball}
can be equivalently written as the following \emph{two-ball form}:
for any $x\in\mathbb R^n$ and $r,s\in(0,\infty)$,
\begin{align}\label{eq-Euclidean-two-ball}
\frac{1}{|B(x,r+s)|}\int_{B(x,r)}w(y)\,dy
\left[
\frac{1}{|B(x,r+s)|}
\int_{B(x,s)}\sigma(y)\,dy
\right]^{p-1}
\lesssim 1,
\end{align}
where the implicit positive constant
is independent of $x$, $r$, and $s$.
Indeed, for any $r,s\in(0,\infty)$,
let $R:=\max\{r,s\}$.
Note that, for any $x\in\mathbb R^n$,
\begin{align*}
B(x,r)\subset B(x,R),
\,B(x,s)\subset B(x,R),
\end{align*}
and
\begin{align*}
|B(x,r+s)|\sim(r+s)^n\sim R^n\sim|B(x,R)|,
\end{align*}
where the positive equivalence constants are independent of
$x$, $r$, and $s$. Thus, applying the ball $A_p$ condition
\eqref{eq-Euclidean-one-ball} with its radius parameter chosen to be $R$,
we obtain \eqref{eq-Euclidean-two-ball}.
Conversely, taking $r=s$
in \eqref{eq-Euclidean-two-ball} and using $|B(x,2r)|\sim |B(x,r)|$
for any $x\in\mathbb{R}^n$ and $r\in(0,\infty)$
yield \eqref{eq-Euclidean-one-ball}. Consequently, \eqref{eq-Euclidean-one-ball}
and \eqref{eq-Euclidean-two-ball} are equivalent.

We now turn to the weight condition \eqref{eq-two-layer-characterization} on $T$.
Obviously, by \eqref{eq-two-layer-characterization},
for $\varepsilon\in(0,\infty)$,  $w\in \mathscr{A}_{p,\varepsilon}$ if and only if,
for any $x\in T$ and $j,m\in\mathbb Z_+$,
\begin{align}\label{eq-twolayer-1}
\frac1{k^{j+m}}\sum_{y\in D_j(x)}w(y)
\left[\frac1{k^{j+m}}\sum_{y\in D_m(x)}
\sigma(y)\right]^{p-1}
\lesssim\frac{1}{k^{\varepsilon(j+m)}}
\end{align}
with the implicit positive constant independent of
$x$, $j$, and $m$. Note that $|D_{j+m}(x)|=k^{j+m}$,
and hence inequality \eqref{eq-twolayer-1}
coincides with the following one
\begin{align}\label{eq-twolayer-2}
\frac1{|D_{j+m}(x)|}\sum_{y\in D_j(x)}w(y)
\left[\frac1{|D_{j+m}(x)|}\sum_{y\in D_m(x)}
\sigma(y)\right]^{p-1}
\lesssim\frac{1}{|D_{j+m}(x)|^\varepsilon}.
\end{align}
Thus, if $\varepsilon=0$ and if we regard the
layers in $T$ as the balls in $\mathbb{R}^n$,
then \eqref{eq-twolayer-2} coincides with \eqref{eq-Euclidean-two-ball}.
This explains the reasonability of $D_{j+m}$ in \eqref{eq-twolayer-2}.

\item[$\mathrm{(i)}_{\mathrm{c}}$]
Moreover, different from the Euclidean ball $A_p$
condition \eqref{eq-Euclidean-two-ball},
the right-hand side of \eqref{eq-twolayer-2}
has the exponential decay $\frac{1}{|D_{j+m}(x)|^\varepsilon}$,
instead of $1$, which is due to the fact that $|D_{j+m}(x)|=k^{j+m}$,
the exponential growth on $j$ and $m$.
We first observe that the analogue of \eqref{eq-twolayer-2}
in the Euclidean setting makes no sense because
even the constant weight $w\equiv1$ does not satisfy such a condition.
To be precise, if the right-hand side of
\eqref{eq-Euclidean-two-ball} is replaced by
$\frac{1}{|B(x,r+s)|^{\varepsilon}}$
for some given $\varepsilon\in(0,\infty)$, then
$w\equiv1$ does not satisfy this new
condition. Indeed, taking $r=s$ in such a condition yields,
for any $x\in\mathbb{R}^n$ and $r\in(0,\infty)$,
\begin{align}\label{eq-1}
\frac{1}{2^{np}}=\frac{|B(x,r)|^p}{|B(x,2r)|^p}
\lesssim\frac{1}{|B(x,2r)|^{\varepsilon}},
\end{align}
which is impossible as $r\to\infty$.
Thus, $w\equiv1$ does not satisfy
\eqref{eq-Euclidean-two-ball}
with its right-hand side replaced by
$\frac{1}{|B(x,r+s)|^{\varepsilon}}$
for any given $\varepsilon\in(0,\infty)$.
In contrast to \eqref{eq-1}, consider again the constant
weight $w\equiv1$ on $T$. Taking $j=m$
in \eqref{eq-twolayer-2}, we obtain,
for any $x\in T$ and $j\in\mathbb{Z}_+$,
\begin{align*}
\frac{1}{k^{jp}}=\frac{|D_j(x)|^p}{|D_{2j}(x)|^p}
\lesssim\frac{1}{|D_{2j}(x)|^\varepsilon}
=\frac{1}{k^{2j\varepsilon}},
\end{align*}
which does not lead to a contradiction when
$\varepsilon\in(0,\frac{p}{2}]$.
Thus, unlike the Euclidean space case, an additional exponential
decay in \eqref{eq-twolayer-2} is admissible.
Furthermore, the parameter $\varepsilon$ in
\eqref{eq-twolayer-2} provides a
stratification of two-layer Muckenhoupt classes.
Moreover, by the definition,
for any $0<\varepsilon_1<\varepsilon_2$,
$\mathscr A_{p,\varepsilon_2}\subset\mathscr A_{p,\varepsilon_1}$,
which further implies that
\begin{align*}
\mathscr{A}_p=
\bigcup_{\varepsilon\in(0,\infty)}\mathscr{A}_{p,\varepsilon}
=\lim_{\varepsilon\in (0,\infty),\, \varepsilon\to 0}\mathscr{A}_{p,\varepsilon}
\end{align*}
in the sense of set convergence.
\end{itemize}

\item[{\rm (ii)}] Let all the notation be the same as in
Definition \ref{def:exponential-Ap}.
Although $\mathscr{A}_{p,\varepsilon}$ is defined
for any $p,\varepsilon\in(0,\infty)$ in Definition \ref{def:exponential-Ap},
we will show in Proposition \ref{prop:nonempty-exp-class}
that $\mathscr{A}_{p,\varepsilon}\ne\emptyset$ precisely when
either $p\in(1,\infty)$ and $\varepsilon\in(0,\frac{p}{2}]$
or $p\in(\frac{1}{2},1]$ and $\varepsilon\in(0,p-\frac{1}{2}]$.
Consequently, $\mathscr{A}_p\ne\emptyset$
if and only if $p\in(\frac{1}{2},\infty)$.
\end{itemize}
\end{remark}

For any $r\in\mathbb{Z}_+$,
the \emph{spherical averaging operator}
$A_r^\circ$ is defined by setting,
for any function $f$ on $T$ and for any $x\in T$,
\begin{align}\label{eq-def-Acirc}
A^\circ_r f(x):=\frac{1}{|S(x,r)|}
\sum_{y\in S(x,r)}|f(y)|.
\end{align}

Our first main result characterizes the
boundedness of $M$ on $L^p(w)$ for $p\in(0,\infty)$.

\begin{theorem}\label{thm-Bound-Ap}
The following two statements hold.
\begin{itemize}
\item[{\rm (I)}] If $p\in(\frac{1}{2},\infty)$,
then, for any weight $w$ on $T$,
the following three assertions are mutually equivalent.
\begin{itemize}
\item[{\rm (i)}] $M$ is bounded on $L^p(w)$.
\item[{\rm (ii)}] $w\in\mathscr{A}_p$.
\item[{\rm (iii)}] There exists $\varepsilon\in(0,\infty)$
such that, for any $r\in\mathbb Z_+$ and $f\in L^p(w)$,
\begin{align}\label{eq-Ar-exponential}
\left\|A_r^\circ f\right\|_{L^p(w)}
\lesssim \frac{1}{k^{\varepsilon r}}\|f\|_{L^p(w)},
\end{align}
where the implicit positive constant is independent of $r$ and $f$.
\end{itemize}
\item[{\rm (II)}] If $p\in(0,\frac{1}{2}]$, then there exists no weight
$w$ on $T$ such that $M$ is bounded on $L^p(w)$.
\end{itemize}
\end{theorem}

\begin{remark}
\begin{itemize}
\item[{\rm (i)}] The key idea used to prove Theorem \ref{thm-Bound-Ap}
in the case $p\in(\frac 12, \infty)$ is as follows.
The boundedness of $M$ in (i) of Theorem \ref{thm-Bound-Ap} first yields a
summability estimate as in \eqref{eq-directional-summability}
for two families of directional averaging operators
$\{P_m\}_{m\in\mathbb{Z}_+}$ and $\{Q_j\}_{j\in\mathbb{Z}_+}$
[see \eqref{eq-def-Pm} and \eqref{eq-def-Qm}].
Using the fact that $\{P_m\}_{m\in\mathbb{Z}_+}$ and $\{Q_j\}_{j\in\mathbb{Z}_+}$
have the semigroup property, we improve the above
summability estimate to the exponential decay of
their operator norms on $L^p(w)$, which is then used to prove (ii) and (iii)
of Theorem \ref{thm-Bound-Ap}.

\item[{\rm (ii)}]
Let $p\in(\frac{1}{2},1]$.
In Example \ref{exam-p<1}, we provide a large family of
weights $w$ on $T$ such that $M$ is bounded on $L^p(w)$.
This, together with Theorem \ref{thm-Bound-Ap}(II),
further implies that the critical value of $p$
such that $M$ is bounded on $L^p(w)$ is $\frac{1}{2}$.
In contrast, the corresponding critical value in $\mathbb{R}^n$ is $1$.
Indeed, there exists no weight $w$ on $\mathbb{R}^n$ such that
$M_{\mathbb{R}^n}$ is bounded on $L^p(w)$ when $p\in(0,1]$.
It also follows from Theorem \ref{thm-Bound-Ap}(II)
and Corollary \ref{cor:exponential-radial}
that $M$ is bounded on $L^p(T)$ if and only if $p\in(1,\infty)$.
Thus, the critical value of $p$
such that $M$ is bounded on the unweighted Lebesgue space $L^p(T)$ is $1$,
which is the same as the corresponding one in $\mathbb{R}^n$.
Thus, for $p\in(\frac{1}{2},1]$, the boundedness of $M$ on $L^p(w)$
is possible only when $w$ is not a constant weight.
\end{itemize}

\end{remark}

Based on Theorem \ref{thm-Bound-Ap}, we find that
$\mathscr{A}_p\neq\emptyset$ if and only if
$p\in (\frac 12, \infty)$. Next, we state the quantitative estimate for the boundedness of $M$ on $L^p(w)$ for any
$p\in(\frac{1}{2},\infty)$. To this end, let
\begin{align*}
\varepsilon^*_p:=\begin{cases}
\displaystyle\frac{p}{2}
&\text{if }p\in(1,\infty),\\
\displaystyle
p-\frac{1}{2}
&\text{if }p\in(\frac{1}{2},1].
\end{cases}
\end{align*}
As a byproduct of the proof of Theorem \ref{thm-Bound-Ap},
we have the following conclusion.

\begin{corollary}\label{cor-exp-quantitative}
Let $p\in(\frac{1}{2},\infty)$, $\varepsilon\in(0,\varepsilon^*_p]$, and
$w\in\mathscr{A}_{p,\varepsilon}$.  Then, for any $f\in L^p(w)$,
\begin{align}\label{eq-exp-quantitative}
\|M f\|_{L^p(w)}\lesssim
[w]_{\mathscr{A}_{p,\varepsilon}}^{\frac{1}{p}}\left\|f\right\|_{L^p(w)},
\end{align}
where the implicit positive constant is independent
of $k$, $f$, and $[w]_{\mathscr{A}_{p,\varepsilon}}$.
Moreover, the exponent $\frac{1}{p}$ of the
weight constant in \eqref{eq-exp-quantitative} is \emph{optimal}.
\end{corollary}

Note that, for $p\in(1,\infty)$, the optimal exponent $\frac{1}{p}$ of
$[w]_{\mathscr{A}_{p,\varepsilon}}$ in \eqref{eq-exp-quantitative}
is different from its Euclidean counterpart $\frac{1}{p-1}$,
which was established by Buckley in \cite[Theorem 2.5]{buc93}.
For any $p\in(1,\infty)$, we establish two further characterizations of the
boundedness of $M$ on $L^p(w)$. For this purpose,
suppose that $w$ is a weight on $T$. Let
\begin{align}\label{eq-def-dualweight}
\sigma:=w^{-\frac{1}{p-1}}.
\end{align}
Then, obviously, $\sigma$ is also a weight on $T$.
For any $E\subset T$,
let $\mathbf{1}_E$ be its \emph{characteristic function}.
Recall that, in $\mathbb{R}^n$, Sawyer \cite{sawyer82}
characterized the two-weight strong-type boundedness of $M_{\mathbb{R}^n}$
by the so-called Sawyer testing condition.
However, Ombrosi and Rivera-R'{\i}os \cite[Theorem 4.3]{or22}
showed that, even in the one-weight setting on $T$,
the classical Sawyer-type testing condition,
namely, for every ball $B\subset T$,
\begin{align}\label{eq-Sawyer-B}
\left\|M\left(\sigma\mathbf{1}_B\right)\mathbf{1}_B
\right\|_{L^p(w)}^p\lesssim\sigma(B)
\end{align}
is insufficient for the boundedness of
$M$ on $L^p(w)$. Next, we state the sufficient
condition for the boundedness of
$M$ on $L^p(w)$ in \cite[Theorem 1.1]{or22}.
For any $E,F\subset T$ and $r\in\mathbb{Z}_+$, let
\begin{align}\label{eq-r-incidence-set}
I_r(E,F):=\{(x,y)\in E\times F:\,d(x,y)=r\}
\end{align}
be the \emph{distance-$r$ incidence set between $E$ and $F$}.
Let $\mathbf{1}\otimes w$ be the weighted
product measure on $T \times T$ determined by
$(\mathbf{1}\otimes w)(E\times F):=|E|w(F)$
for $E,F\subset T$. In \cite[Theorem 1.1]{or22}, Ombrosi and
Rivera-R\'{\i}os proved that, if there exists $\beta\in(0,1)$
and $\alpha\in(\beta,p)$ such that,
for any $E,F\subset T$ and $r\in\mathbb{Z}_+$,
\begin{align}\label{eq-condition-sufficient}
(\mathbf{1}\otimes w)\left(I_r(E,F)\right)
\lesssim k^{\beta r}\left[w(E)\right]^{\frac{\alpha}{p}}
\left[w(F)\right]^{1-\frac{\alpha}{p}},
\end{align}
where the implicit positive constant is independent of
$E$, $F$, and $r$, then $M$ is bounded on $L^p(w)$.
Motivated by these results, we obtain the following two further characterizations
of the boundedness of $M$ on $L^p(w)$ when $p\in(1,\infty)$.

\begin{theorem}\label{thm-strong}
Let $p\in(1,\infty)$,  $w$ be a weight on $T$, and
$\sigma$ as in \eqref{eq-def-dualweight}.  Then the following
four statements are mutually equivalent.
\begin{itemize}
\item[{\rm (i)}] $M$ is bounded on $L^p(w)$.
\item[{\rm (ii)}] For any ball $B\subset T$,
\begin{align}\label{eq-global-ball-testing}
\left\|M\left(\sigma\mathbf{1}_B\right)\right\|_{L^p(w)}^p
\lesssim\sigma(B),
\end{align}
where the implicit positive constant is independent of $B$.
\item[{\rm (iii)}] There exists $\beta\in(0,1)$ such that, for any
$E,F\subset T$ and $r\in\mathbb Z_+$,
\begin{align}\label{eq-MA-alpha-one}
(\mathbf{1}\otimes w)\left(I_r(E,F)\right)
\lesssim k^{\beta r}[w(E)]^{\frac1p}[w(F)]^{1-\frac1p},
\end{align}
where the implicit positive constant is independent of $E$, $F$, and $r$.
\item[{\rm (iv)}] $w\in\mathscr{A}_p$.
\end{itemize}
\end{theorem}

\begin{remark}
By Theorem \ref{thm-strong} and \cite[Theorem 4.3]{or22}, we find that
the local Sawyer-type testing condition in \eqref{eq-Sawyer-B}
does not characterize the boundedness of $M$ on $L^p(w)$,
but its global counterpart in \eqref{eq-global-ball-testing} does.
Moreover, it follows from Theorem \ref{thm-strong} that the
estimate \eqref{eq-condition-sufficient} for the weighted
product measure of incidence sets is also necessary for
the boundedness of $M$ on $L^p(w)$ in the sense that this boundedness guarantees that
there exists some $\beta\in(0,1)$ such that \eqref{eq-condition-sufficient} holds
with $\alpha=1$.
\end{remark}

The next theorem shows that,
for any $p\in(\frac{1}{2},1]$, the boundedness
of $M$ on $L^p(w)$ is equivalent to the following
point-testing condition.

\begin{theorem}\label{thm-atomic-characterization}
Let $p\in(\frac{1}{2},1]$, and let $w$ be a weight on $T$.
Then $M$ is bounded on $L^p(w)$ if and only if
\begin{align}\label{eq-test-onepoint}
\sup_{y\in T}\frac1{w(y)}
\sum_{x\in T}\frac{w(x)}{k^{p d(x,y)}}<\infty
\end{align}
if and only if $w\in\mathscr{A}_p$.
\end{theorem}

The remainder of this article is organized as follows.

In Section \ref{sec:weight}, we study the basic properties
of the classes $\mathscr{A}_p$,
including the non-emptiness, the strict monotonicity, and the open property
(see Propositions \ref{prop:nonempty-exp-class} and \ref{prop:exp-class-structure}).

Section \ref{sec:strong} is devoted to the proofs of
Theorems \ref{thm-Bound-Ap}, \ref{thm-strong}, and \ref{thm-atomic-characterization}
and Corollary \ref{cor-exp-quantitative}. For this purpose, we first establish the
relationship between the operator norms of composition operators
$\{Q_j\circ P_m\}_{j,m\in\mathbb{Z}_+}$ on $L^p(w)$ and the weight constant
$[w]_{\mathscr{A}_{p,\varepsilon}}$
in Definition \ref{def:exponential-Ap} (see Lemma \ref{lem-exact-two-layer-norm}).
Using Lemma \ref{lem-directional-decay}, we show that the boundedness of
$M$ on $L^p(w)$ guarantees the exponential decay of
the operator norms of $\{P_m\}_{m\in\mathbb{Z}_+}$
and $\{Q_j\}_{j\in\mathbb{Z}_+}$ on $L^p(w)$.
We then prove Theorem \ref{thm-Bound-Ap} and Corollary \ref{cor-exp-quantitative}.
As a corollary of Theorem \ref{thm-Bound-Ap}, we also
obtain the summability characterizations of
the boundedness of $M$ on $L^p(w)$ (see Corollary \ref{cor-sum-characterization}).
Next, we show that, when $p\in(1,\infty)$, \eqref{eq-global-ball-testing}
also yields the exponential decay of operator norms of $\{P_m\}_{m\in\mathbb{Z}_+}$
and $\{Q_j\}_{j\in\mathbb{Z}_+}$ on $L^p(w)$ (see Lemma \ref{lem-global-testing-decay}).
Using this, we prove Theorem \ref{thm-strong}. Finally,
we show Theorem \ref{thm-atomic-characterization}.

In Section \ref{sec:radial}, for any $p\in(\frac{1}{2},\infty)$,
we obtain the characterizations of the
boundedness of $M$ on $L^p(w)$ and from $L^p(w)$ to $L^{p,\infty}(w)$
for radial weights $w$ (see Theorem \ref{thm-bound-radial}).
Furthermore, when $p\in(0,\frac{1}{2}]$, we also prove that
there exists no weight $w$ on $T$ such that
$M$ is bounded from $L^p(w)$ to $L^{p,\infty}(w)$
(see Proposition \ref{prop:weak-lower-endpoint}).
In Proposition \ref{prop:weak-atomic-characterization},
for any $p\in(\frac{1}{2},1)$, we obtain the characterization
of the boundedness of $M$ from $L^p(w)$ to $L^{p,\infty}(w)$.
Moreover, for any $p\in(\frac{1}{2}, \infty)$,
we determine the exact ranges
of $\beta\in\mathbb{R}$ such that $M$ is bounded on $L^p(w_\beta)$ and
from $L^p(w_\beta)$ to $L^{p,\infty}(w_\beta)$, where
$w_\beta$ is the exponential radial weight in \eqref{eq-def-radial-beta}
(see Corollary \ref{cor:exponential-radial}).

Finally, in Section \ref{sec:applications}, we give several
applications of the exponential decay obtained in Theorem \ref{thm-Bound-Ap},
including the boundedness results for exponentially decaying kernel operators
(see Theorem \ref{thm:decaying-kernel}),
the pointwise exponential decay of spherical and ball averages
(see Proposition \ref{prop:pointwise-decay-averages}),
and the weighted Fefferman--Stein vector-valued
inequalities (see Theorem \ref{thm:vector-valued}).

We end this introduction with some conventions on the notation.
Let $\mathbb{Z}$ denote the set of all integers.
For any $p\in(1,\infty)$,
let $p':=\frac{p}{p-1}$ denote the \emph{conjugate index} of $p$
(that is, $\frac{1}{p}+\frac{1}{p'}=1$) and,
for any $a\in\mathbb{R}$, let $\lfloor a\rfloor$ denote the
largest integer not greater than $a$, and let $a_+:=\max\{a,0\}$.
Suppose that $p\in(0,\infty)$ and $w$ is a weight on $T$.
For an operator $\mathscr{T}$ bounded on $L^p(w)$, let
$\|\mathscr{T}\|_{L^p(w)\to L^p(w)}$ denote its operator norm.
The symbol $C$ denotes a positive constant which is independent
of the main parameters involved, but may vary from line to line.
The notation $A\lesssim B$ means that $A\leq CB$ for some positive constant $C$,
while $A\sim B$ means $A\lesssim B\lesssim A$.
Finally, in all subsequent proofs we retain the notation introduced
in the relevant statement.

\section{Two-Layer Muckenhoupt Weights}
\label{sec:weight}

In this section, we investigate the properties of
the classes $\mathscr{A}_{p,\varepsilon}$ and
$\mathscr{A}_p$
in Definition \ref{def:exponential-Ap}, including the
precise ranges of $p$ and $\varepsilon$ such that these classes are
nonempty and the strict monotonicity, and the open property of
the classes $\mathscr{A}_p$.

We begin with some notation.
For any $N\in\mathbb{Z}_+$, let
\begin{align}\label{eq-def-TN}
T_N:=\left\{x\in T:\,|x|=N\right\}.
\end{align}
For any $\beta\in\mathbb{R}$, the \emph{exponential
radial weight $w_\beta$} is defined by setting, for any $x\in T$,
\begin{align}\label{eq-def-radial-beta}
w_\beta(x):=k^{\beta|x|}.
\end{align}
Then we have the following nonempty properties of
$\mathscr{A}_{p,\varepsilon}$ and $\mathscr{A}_p$.
\begin{proposition}\label{prop:nonempty-exp-class}
Let $p,\varepsilon\in(0,\infty)$.  Then
$\mathscr{A}_{p,\varepsilon}\ne\emptyset$ if and only if
both $p\in(1,\infty)$ and $\varepsilon\in(0,\frac{p}{2}]$ or
both $p\in(\frac{1}{2},1]$ and $\varepsilon\in(0,p-\frac{1}{2}]$.
Consequently, $\mathscr{A}_p\ne\emptyset$
if and only if $p\in(\frac{1}{2},\infty)$.
\end{proposition}

\begin{proof}
We first prove the necessity by considering
two cases for the range of $p$.

\emph{Case (1)} $p\in(1,\infty)$. In this case,
suppose that $w\in\mathscr{A}_{p,\varepsilon}$.
Using H\"older's inequality, we conclude that,
for any $x\in T$ and $n\in\mathbb Z_+$,
\begin{align*}
w(D_n(x))\left[w^{-\frac{1}{p-1}}(D_n(x))
\right]^{p-1}\geq |D_n(x)|^p=k^{pn}.
\end{align*}
Taking $j=m=n$ in \eqref{eq-two-layer-characterization},
we obtain $k^{n(2\varepsilon-p)}\leq[w]_{\mathscr{A}_{p,\varepsilon}}<\infty$.
It follows from the arbitrariness of $n\in\mathbb{Z}_+$ that
$\varepsilon\in(0,\frac{p}{2}]$.
This completes the necessity in this case.

\emph{Case (2)} $p\in(0,1]$.
In this case, assume that $w\in\mathscr{A}_{p,\varepsilon}$.
Observe that, for any $x\in T$ and $n\in\mathbb Z_+$,
\begin{align*}
w(D_n(x))\max_{y\in D_n(x)}w^{-1}(y)\geq |D_n(x)|=k^n.
\end{align*}
Taking $j=m=n$ again in \eqref{eq-two-layer-characterization},
we find that $k^{n(1-2p+2\varepsilon)}\leq
[w]_{\mathscr{A}_{p,\varepsilon}}<\infty$.
Since $n\in\mathbb{Z}_+$ is arbitrary, we obtain
$p\in(\frac{1}{2},1]$ and $\varepsilon\in(0,p-\frac{1}{2}]$,
which completes the proof of the necessity.

Next, we show the sufficiency. To prove that $\mathscr{A}_{p,\varepsilon}$
is nonempty in the desired ranges, for any  $\beta\in\mathbb R$
let $w_\beta$ be as in \eqref{eq-def-radial-beta}.
Note that, for any $N,m\in\mathbb{Z}_+$ and $x\in T_N$,
\begin{align*}
w_\beta(D_m(x))=k^{m+\beta(N+m)} \text{ and }
\max_{y\in D_m(x)}w_\beta^{-1}(y)=\frac{1}{k^{\beta(N+m)}}.
\end{align*}
Using these and \eqref{eq-two-layer-characterization}
shows that, for any $\varepsilon\in(0,\infty)$,
\begin{align}\label{eq-radial-constant}
[w_\beta]_{\mathscr{A}_{p,\varepsilon}}
&=\begin{cases}
\displaystyle\sup_{\genfrac{}{}{0pt}{}{x\in T}{j,m\in\mathbb Z_+}}
\frac{w_\beta(D_j(x))[w_\beta^{-\frac{1}{p-1}}(D_m(x))]^{p-1}}
{k^{(p-\varepsilon)(j+m)}} &\text{ if } p\in(1,\infty),\\
\displaystyle\sup_{\genfrac{}{}{0pt}{}{x\in T}{j,m\in\mathbb Z_+}}
\frac{w_\beta(D_j(x))\displaystyle \max_{y\in D_m(x)}w_\beta^{-1}(y)}{k^{(p-\varepsilon)(j+m)}}
&\text{ if } p\in\left(\frac{1}{2},1\right]
\end{cases}\nonumber\\
&=\begin{cases}
\displaystyle\sup_{j,m\in\mathbb Z_+}
k^{j(1+\beta-p+\varepsilon)}k^{m(\varepsilon-1-\beta)} &\text{ if }  p\in(1,\infty),\\
\displaystyle\sup_{j,m\in\mathbb Z_+}
k^{j(1+\beta-p+\varepsilon)}k^{m(\varepsilon-p-\beta)} &\text{ if } p\in\left(\frac{1}{2},1\right]
\end{cases}
\end{align}
and hence
\begin{align}\label{eq-rangeradial-p>1}
w_\beta\in\mathscr{A}_{p,\varepsilon}\text{ if and only if }
\begin{cases}
\beta\in[\varepsilon-1,p-1-\varepsilon] &\text{ if } p\in(1,\infty),\\
\beta\in[\varepsilon-p, p-1-\varepsilon] &\text{ if } p\in\left(\frac{1}{2},1\right].
\end{cases}
\end{align}
Clearly, when both $p\in(1,\infty)$ and
$\varepsilon\in(0,\frac{p}{2}]$, or both
$p\in(\frac{1}{2},1]$ and $\varepsilon
\in(0,p-\frac{1}{2}]$, the ranges of $\beta$ in \eqref{eq-rangeradial-p>1} are nonempty and hence the class $\mathscr{A}_{p,\varepsilon}$ is nonempty.
This completes the proof of the sufficiency.

Moreover, it follows from the definition of $\mathscr{A}_p$
that  $\mathscr{A}_p\ne\emptyset$
if and only if $p\in(\frac{1}{2},\infty)$.
This completes the proof of Proposition \ref{prop:nonempty-exp-class}.
\end{proof}

In the following proposition, we obtain the
strict monotonicity and the open property of
the classes $\mathscr{A}_p$.

\begin{proposition}\label{prop:exp-class-structure}
The following two assertions hold.
\begin{itemize}
\item[{\rm (i)}]
If $\frac12<p<q<\infty$, then
$\mathscr{A}_p\subsetneqq\mathscr{A}_q$.

\item[{\rm (ii)}]
For any $p\in(\frac12,\infty)$,
\begin{align}\label{eq-open-exp-class}
\mathscr{A}_p
=\bigcup_{r\in(\frac{1}{2},p)}
\mathscr{A}_r.
\end{align}
\end{itemize}
\end{proposition}

\begin{proof}
Suppose that $w$ is a weight on $T$.
For any $u\in(\frac12,\infty)$ and any finite set $E\subset T$,
let $\theta(u):=(u-1)_+$ and
\begin{align*}
\Phi_u(E,w):=
\begin{cases}
\displaystyle\left[\sum_{y\in E}
w^{-\frac1{u-1}}(y)\right]^{u-1}
&\text{if }u\in(1,\infty),\\
\displaystyle\max_{y\in E}w^{-1}(y)
&\text{if }u\in(\frac12,1].
\end{cases}
\end{align*}

We first prove \textup{(i)}. Let $\frac12<p<q<\infty$ and suppose
that $w\in\mathscr{A}_p$. By the definition of $\mathscr{A}_p$,
we find that there exists $\varepsilon\in(0,\infty)$ such that
$w\in\mathscr{A}_{p,\varepsilon}$.
Recall the following fundamental inequality:
for any $\gamma\in(0,1]$ and any
sequence $\{a_n\}_{n\in\mathbb{N}}$ of non-negative numbers,
\begin{align}\label{eq-embeding-l1}
\left(\sum_{n\in\mathbb{N}}a_n\right)^{\gamma}
\leq\sum_{n\in\mathbb{N}}a^\gamma_n.
\end{align}
Since $|D_m(x)|=k^m$ for any $x\in T$ and $m\in\mathbb{Z}_+$,
it follows from H\"older's inequality and \eqref{eq-embeding-l1} that
\begin{align*}
\Phi_q(D_m(x),w)\leq
k^{[\theta(q)-\theta(p)]m}\Phi_p(D_m(x),w).
\end{align*}
Observe that $0\leq\theta(q)-\theta(p)\leq q-p$.
Consequently, for any $x\in T$ and $j,m\in\mathbb Z_+$,
\begin{align*}
\frac{w(D_j(x))\Phi_q(D_m(x),w)}{k^{(q-\varepsilon)(j+m)}}
&\leq\frac{w(D_j(x))\Phi_p(D_m(x),w)}{k^{(p-\varepsilon)(j+m)}}
\frac{1}{k^{(q-p)j+\{q-p-[\theta(q)-\theta(p)]\}m}}\\
&\leq\frac{w(D_j(x))\Phi_p(D_m(x),w)}{k^{(p-\varepsilon)(j+m)}}.
\end{align*}
Taking the supremum over all $x$, $j$, and $m$ on both sides, we obtain
$[w]_{\mathscr{A}_{q,\varepsilon}}\leq[w]_{\mathscr{A}_{p,\varepsilon}}$,
and hence $w\in \mathscr{A}_{q,\varepsilon}$.
Thus, $\mathscr{A}_p\subset\mathscr{A}_q$.
To complete the proof of (i), it remains to prove that
this inclusion is proper. For any $\beta\in \mathbb{R}$,
let $w_\beta$ be as in \eqref{eq-def-radial-beta}.
Using \eqref{eq-rangeradial-p>1}, we conclude that, for any
$u\in(\frac12,\infty)$, $w_\beta\in\mathscr{A}_u$
if and only if $\beta\in(-\min\{u,1\},u-1)$.
Thus, taking $\beta:=p-1$ we obtain
$w_{\beta}\notin\mathscr{A}_p$.
On the other hand, since $p\in(\frac12,q)$,
from \eqref{eq-rangeradial-p>1}, we deduce that
$w_{p-1}\in\mathscr{A}_q$.
Therefore, $\mathscr{A}_p\subsetneqq
\mathscr{A}_q$. This completes the proof of (i).

We next prove \textup{(ii)}. By (i), we find that, for any
$r\in(\frac12,p)$, $\mathscr{A}_r
\subset\mathscr{A}_p$, and hence
\begin{align}\label{eq-open-reverse-inclusion}
\bigcup_{r\in(\frac{1}{2},p)}\mathscr{A}_r
\subset\mathscr{A}_p.
\end{align}
To show the reverse inclusion in \eqref{eq-open-reverse-inclusion},
let $w\in\mathscr{A}_p$. From the definition of $\mathscr{A}_p$,
we infer that there exists $\varepsilon\in(0,\infty)$ such that
$w\in\mathscr{A}_{p,\varepsilon}$.
Take $r\in(\max\{\frac12,p-\varepsilon\},p)$,
and let $\varepsilon_r:=\varepsilon-(p-r)$.
Note that the choice of $r$ guarantees that
$\varepsilon_r\in(0,\infty)$ and
$r-\varepsilon_r=p-\varepsilon$.
Moreover, since $r<p$, using \eqref{eq-embeding-l1},
we conclude that, for any $x\in T$ and $j,m\in\mathbb Z_+$,
\begin{align*}
\Phi_r(D_m(x),w)\leq\Phi_p(D_m(x),w),
\end{align*}
and hence
\begin{align*}
\frac{w(D_j(x))\Phi_r(D_m(x),w)}{k^{(r-\varepsilon_r)(j+m)}}
=\frac{w(D_j(x))\Phi_r(D_m(x),w)}{k^{(p-\varepsilon)(j+m)}}
\leq\frac{w(D_j(x))\Phi_p(D_m(x),w)}{k^{(p-\varepsilon)(j+m)}}.
\end{align*}
Taking the supremum over all $x$, $j$, and $m$, we obtain
$[w]_{\mathscr{A}_{r,\varepsilon_r}}\leq
[w]_{\mathscr{A}_{p,\varepsilon}}<\infty$,
which further implies that $w\in\mathscr{A}_{r,\varepsilon_r}
\subset\mathscr{A}_r$ and hence
\begin{align}\label{eq-open-forward-inclusion}
\mathscr{A}_p\subset
\bigcup_{r\in(\frac{1}{2},p)}\mathscr{A}_r.
\end{align}
Combining \eqref{eq-open-reverse-inclusion} and
\eqref{eq-open-forward-inclusion}, we conclude that
\eqref{eq-open-exp-class} holds.
This completes the proof of (ii) and
hence Proposition \ref{prop:exp-class-structure}.
\end{proof}

\section{Directional Operators and Strong-Type Characterizations}\label{sec:strong}

This section is devoted to the proofs of the main results
presented in Section \ref{sec:intro}.
To be precise, we present the proofs of Theorem
\ref{thm-Bound-Ap} and Corollary \ref{cor-exp-quantitative}
in Subsection \ref{sub3-1}, and the proofs of
Theorems \ref{thm-strong} and \ref{thm-atomic-characterization}
in Subsection \ref{sub3-2}.
The main idea is to introduce two directional
averaging operators and use the branching
structure of the tree to obtain the exponential
decay of their operator norms.

\subsection{Proofs of Theorem \ref{thm-Bound-Ap} and Corollary \ref{cor-exp-quantitative}}\label{sub3-1}

In this subsection, we prove Theorem \ref{thm-Bound-Ap}
and Corollary \ref{cor-exp-quantitative}.
Suppose that $k\geq 2$ is an integer
and that $T$ is an infinite rooted $k$-ary tree.
The \emph{spherical maximal
operator} $M^\circ$ is defined by setting,
for any function $f$ on $T$ and for any $x\in T$,
\begin{align}\label{eq-def-Mcicr}
M^\circ f(x):=\sup_{r\in\mathbb Z_+}A_r^\circ f(x)
=\sup_{r\in\mathbb Z_+}\frac{1}{|S(x,r)|}\sum_{y\in S(x,r)}|f(y)|.
\end{align}
By \cite[Proposition 2.1]{ors21},
we find that, for any function $f$ on $T$ and for any $x\in T$,
\begin{align}\label{eq-equivalent-M}
Mf(x)\leq M^{\circ}f(x)\leq2 Mf(x).
\end{align}
Therefore, the weighted boundedness properties of $M$ and
$M^\circ$ are equivalent, and it suffices to study
$M^\circ$ and the spherical averaging operators
$\{A_r^\circ\}_{r\in\mathbb Z_+}$.
We next introduce the directional averaging operators.
For any $z,y\in T$, if $z$ lies on the unique
path connecting $y$ to $o$,
we say that $z$ is an \emph{ancestor} of $y$ and $y$ is a \emph{descendant} of $z$.
For any $x\in T$ and $j\in\mathbb{Z}_+\cap[0,|x|]$,
the notation $x^{(j)}$
denotes the ancestor of $x$ such that $d(x^{(j)},x)=j$.
Note that $x^{(0)}=x$ and $x^{(|x|)}=o$.
For any $m,j\in\mathbb Z_+$, the \emph{descendant averaging
operator} $P_m$ and the \emph{ancestor averaging operator} $Q_j$
are defined, respectively, by setting,
for any function $f$ on $T$ and for any $x\in T$,
\begin{align}\label{eq-def-Pm}
P_mf(x):=\frac{1}{k^m}\sum_{y\in D_m(x)}|f(y)|
\end{align}
and
\begin{align}\label{eq-def-Qm}
Q_j f(x):=
\begin{cases}
\displaystyle\frac{1}{k^j}\left|f\left(x^{(j)}\right)\right|&\text{ if }j\leq |x|,\\
0&\text{ if }j>|x|.
\end{cases}
\end{align}
For any $j,m\in\mathbb Z_+$, we define their
composition operator $T_{j,m}$ by setting,
for any function $f$ on $T$ and for any $x\in T$,
\begin{align*}
T_{j,m}f(x):=Q_j(P_mf)(x)=
\begin{cases}
\displaystyle \frac{1}{k^{j+m}}\sum_{y\in D_m(x^{(j)})}|f(y)|
&\text{ if }j\leq |x|,\\
0
&\text{ if }j> |x|.
\end{cases}
\end{align*}

Before proving Theorem \ref{thm-Bound-Ap} and Corollary \ref{cor-exp-quantitative},
we need to study the boundedness properties of the directional averaging operators
and their compositions. The following lemma
gives a pointwise estimate of
$\{A^{\circ}_r\}_{r\in\mathbb{Z}_+}$ in terms of
$\{T_{j,m}\}_{j,m\in\mathbb{Z}_+}$.
For any $x,y\in T$, their \emph{last common ancestor},
denoted by $x\wedge y$, is the common ancestor of both
$x$ and $y$ which has the largest distance from $o$.

\begin{lemma}\label{lem:strong-pointwise}
Suppose that $f$ is a function on $T$.
Then, for any $r\in\mathbb Z_+$ and $x\in T$,
\begin{align}\label{eq-Ar-QP}
A_r^\circ f(x)\lesssim\sum_{\genfrac{}{}{0pt}{}
{j,m\in\mathbb Z_+}{j+m=r}}T_{j,m}f(x),
\end{align}
where the implicit positive constant is independent of $r$, $f$, and $x$.
\end{lemma}

\begin{proof}
Suppose that $r\in\mathbb{Z}_+$ and $x\in T$.
For any $y\in S(x,r)$, let
\begin{align*}
a:=x\wedge y\in T,
\  j:=d(x,a)\in\mathbb{Z}_+,\
\text{ and }
m:=d(a,y)\in\mathbb{Z}_+.
\end{align*}
It follows from the structure of $T$ that
$j\in[0, \min\{r,|x|\}]$.
Since $j+m=r$ and $a=x^{(j)}$, we deduce that
$y\in D_{r-j}(x^{(j)})$.
By this and \eqref{eq-sphere-ball-growth}, we find that
\begin{align*}
A_r^\circ f(x)&=\frac{1}{|S(x,r)|}
\sum_{y\in S(x,r)}|f(y)|
\lesssim \frac{1}{k^{r}}\sum_{j=0}^{\min\{r,|x|\}}
\sum_{y\in D_{r-j}(x^{(j)})}|f(y)|\\
&=\sum_{j=0}^{\min\{r,|x|\}}\frac{1}{k^{j}}
\left[\frac{1}{k^{r-j}}\sum_{y\in D_{r-j}(x^{(j)})}|f(y)|\right]
=\sum_{j=0}^{\min\{r,|x|\}}Q_j(P_{r-j}f)(x)\\
&=\sum_{\genfrac{}{}{0pt}{}{j,m\in\mathbb Z_+}{j+m=r}}
Q_j(P_mf)(x)=\sum_{\genfrac{}{}{0pt}{}
{j,m\in\mathbb Z_+}{j+m=r}}T_{j,m}f(x).
\end{align*}
This completes the proof of Lemma \ref{lem:strong-pointwise}.
\end{proof}

The following lemma establishes the relationship
between the operator norms of $\{T_{j,m}\}_{j,m\in\mathbb{Z}_+}$
and the weight constant in Definition \ref{def:exponential-Ap}.

\begin{lemma}\label{lem-exact-two-layer-norm}
Let $p\in(0,\infty)$, and let $w$ be a weight on $T$.
Then, for any $j,m\in\mathbb Z_+$,
\begin{align}\label{eq-exact-two-layer-norm}
\left\|T_{j,m}\right\|_{L^p(w)\to L^p(w)}^p
=\begin{cases}
\frac{\displaystyle\sup_{x\in T}
w(D_j(x))[w^{-\frac{1}{p-1}}(D_m(x))]^{p-1}}{\displaystyle k^{p(j+m)}}
&\text{if }p\in(1,\infty),\\
\displaystyle\frac{\displaystyle \sup_{x\in T}w(D_j(x))
\max_{y\in D_m(x)}w^{-1}(y)}{k^{p(j+m)}}&\text{if }p\in(0,1].
\end{cases}
\end{align}
Furthermore, for any $\varepsilon\in(0,\infty)$,
\begin{align}\label{eq-characteristic-operator-form}
[w]_{\mathscr{A}_{p,\varepsilon}}=\sup_{j,m\in\mathbb Z_+}
k^{\varepsilon(j+m)}\left\|T_{j,m}\right\|_{L^p(w)\to L^p(w)}^p.
\end{align}
\end{lemma}

\begin{proof}
Let $j,m\in\mathbb Z_+$. To prove \eqref{eq-exact-two-layer-norm},
from a change of variables, we infer that, for any $f\in L^p(w)$,
\begin{align}\label{eq-two-layer-basic-identity}
\left\|T_{j,m}f\right\|_{L^p(w)}^p
=\frac{1}{k^{p(j+m)}}\sum_{x\in T}w(D_j(x))
\left[\sum_{y\in D_m(x)}|f(y)|\right]^p.
\end{align}

When $p\in(1,\infty)$, for brevity we write
$$\Gamma_{j,m,x}:=w(D_j(x))\left[w^{-\frac{1}{p-1}}(D_m(x))\right]^{p-1}$$
for any $x\in T$. Using H\"older's inequality and the fact that
the layers $\{D_m(x)\}_{x\in T}$ are pairwise disjoint, we conclude that
\begin{align*}
\left\|T_{j,m}f\right\|_{L^p(w)}^p
&\leq\frac{1}{k^{p(j+m)}}\sum_{x\in T}\Gamma_{j,m,x}
\sum_{y\in D_m(x)}|f(y)|^pw(y)\nonumber\\
&\leq\frac{1}{k^{p(j+m)}}\sup_{z\in T}\Gamma_{j,m,z}
\sum_{x\in T}\sum_{y\in D_m(x)}|f(y)|^pw(y)\nonumber\\
&\leq\frac{1}{k^{p(j+m)}}\sup_{x\in T}\Gamma_{j,m,x}
\left\|f\right\|_{L^p(w)}^p,
\end{align*}
and hence
\begin{align}\label{eq-Tjm-leq}
\left\|T_{j,m}\right\|_{L^p(w)\to L^p(w)}^p\leq
\frac{1}{k^{p(j+m)}}\sup_{x\in T}\Gamma_{j,m,x}
\end{align}
For any $x\in T$ and $m\in\mathbb{Z}_+$, letting $f:=\sigma\mathbf1_{D_m(x)}$
in \eqref{eq-two-layer-basic-identity}, we find that
\begin{align*}
\left\|T_{j,m}\right\|_{L^p(w)\to L^p(w)}^p\geq
\frac{\|T_{j,m}f\|_{L^p(w)}^p}{\|f\|_{L^p(w)}^p}
=\frac{1}{k^{p(j+m)}}\Gamma_{j,m,x}.
\end{align*}
Taking the supremum over all $x\in T$, we obtain the
reverse estimate of \eqref{eq-Tjm-leq}. This, combined with
\eqref{eq-Tjm-leq}, completes the proof of
\eqref{eq-exact-two-layer-norm} in the case $p\in(1,\infty)$.

When $p\in(0,1]$,  by \eqref{eq-embeding-l1}
and \eqref{eq-two-layer-basic-identity}, we find that
\begin{align*}
\left\|T_{j,m}f\right\|_{L^p(w)}^p
&\leq\frac{1}{k^{p(j+m)}}\sum_{x\in T}w(D_j(x))
\sum_{y\in D_m(x)}|f(y)|^p\\
&\leq\frac{1}{k^{p(j+m)}}\sup_{x\in T}
\left[w(D_j(x))\max_{y\in D_m(x)}w^{-1}(y)\right]
\sum_{x\in T}\sum_{y\in D_m(x)}|f(y)|^pw(y)\\
&\leq\frac{1}{k^{p(j+m)}}\sup_{x\in T}
\left[w(D_j(x))\max_{y\in D_m(x)}w^{-1}(y)\right]
\left\|f\right\|_{L^p(w)}^p,
\end{align*}
where, in the last step, we use the fact that
the layers $\{D_m(x)\}_{x\in T}$ are pairwise disjoint again. Thus,
\begin{align}\label{eq-Tjm-leq-1}
\left\|T_{j,m}\right\|_{L^p(w)\to L^p(w)}^p\leq
\frac{1}{k^{p(j+m)}}\sup_{x\in T}w(D_j(x))
\max_{y\in D_m(x)}w^{-1}(y).
\end{align}
The reverse inequality of \eqref{eq-Tjm-leq-1}
follows from testing $\mathbf1_{\{y\}}$ in \eqref{eq-two-layer-basic-identity},
where $y$ is a point of minimum weight in $D_m(x)$; we omit the details.
This proves \eqref{eq-exact-two-layer-norm} in the case $p\in(0,1]$.
Moreover, by \eqref{eq-two-layer-characterization}, we find that
\eqref{eq-characteristic-operator-form} holds.
This completes the proof of Lemma \ref{lem-exact-two-layer-norm}.
\end{proof}

In the following lemma, we pass from the summability
estimate of the directional averages to the exponential decay
of their operator norms on $L^p(w)$.

\begin{lemma}\label{lem-directional-decay}
Let $p\in(0,\infty)$, $w$ be a weight on $T$, and
$\{R_m\}_{m\in\mathbb Z_+}$ denote either $\{P_m\}_{m\in\mathbb Z_+}$
or $\{Q_m\}_{m\in\mathbb Z_+}$.  Suppose that there exists a
positive constant $A\in(1,\infty)$ such that, for any $f\in L^p(w)$,
\begin{align}\label{eq-directional-summability}
\sum_{m\in\mathbb Z_+}\|R_mf\|_{L^p(w)}^p
\leq A\|f\|_{L^p(w)}^p.
\end{align}
Then there exists a positive constant $\theta\in(0,1)$
such that, for any $m\in\mathbb Z_+$ and $f\in L^p(w)$,
\begin{align}\label{eq-directional-exp-general}
\|R_mf\|_{L^p(w)}\leq A^{\frac{1}{p}}\theta^m\|f\|_{L^p(w)}.
\end{align}
\end{lemma}

\begin{proof}
Let $f\in L^p(w)$. For any $n\in\mathbb Z_+$, let
$$
\Phi_n(f):=\sum_{\ell\in\mathbb Z_+}
\left\|R_{\ell+n}f\right\|_{L^p(w)}^p.
$$
From \eqref{eq-def-Pm} and \eqref{eq-def-Qm}, we deduce that,
for any $\ell, n\in\mathbb{Z}_+$, $R_\ell(R_nf)=R_{\ell+n}f$.
For any $n\in\mathbb{Z}_+$,
applying \eqref{eq-directional-summability} to $R_nf$ we obtain
 $\Phi_n(f)\leq A\left\|R_nf\right\|_{L^p(w)}^p$.
Since, for any $n\in\mathbb{Z}_+$,
$$
\Phi_n(f)=\|R_nf\|_{L^p(w)}^p+\Phi_{n+1}(f),
$$
it follows that
\begin{align}\label{eq-decay}
\Phi_{n+1}(f)\leq\left(1-\frac{1}{A}\right)\Phi_n(f).
\end{align}
Let $\theta:=(1-\frac{1}{A})^{\frac{1}{p}}\in(0,1)$.
Iterating estimate \eqref{eq-decay}
and using \eqref{eq-directional-summability},
we conclude that, for any $n\in\mathbb{Z}_+$,
$$
\|R_nf\|_{L^p(w)}^p\leq \Phi_n(f)
\leq \theta^{np}\Phi_0(f)\leq A\theta^{np}\|f\|_{L^p(w)}^p.
$$
Taking the $p$-th root of both sides then completes the
proof of \eqref{eq-directional-exp-general},
and hence Lemma \ref{lem-directional-decay}.
\end{proof}

The following lemma shows that, for $w\in \mathscr{A}_{p,\varepsilon}$,
the operator norms of $\{A_r^\circ\}_{r\in\mathbb Z_+}$ on $L^p(w)$
decay exponentially on $r$.

\begin{lemma}\label{lem:exp-class-synthesis}
Let $p\in(0,\infty)$, $\varepsilon\in(0,\infty)$, and
$w\in\mathscr{A}_{p,\varepsilon}$.
Then, for any $r\in\mathbb Z_+$,
\begin{align*}
\left\|A_r^\circ\right\|^p_{L^p(w)\to L^p(w)}
\lesssim [w]_{\mathscr{A}_{p,\varepsilon}}
(r+1)^{\max\{1,p\}}\frac{1}{k^{\varepsilon r}},
\end{align*}
where the implicit positive constant is
independent of $k$, $\varepsilon$, and $r$.
\end{lemma}

\begin{proof}
From Lemma \ref{lem-exact-two-layer-norm}, we infer that,
for any $j,m\in\mathbb{Z}_+$,
\begin{align*}
\left\|T_{j,m}\right\|_{L^p(w)\to L^p(w)}^p\leq
[w]_{\mathscr{A}_{p,\varepsilon}}
\frac{1}{k^{\varepsilon(j+m)}}.
\end{align*}
If $p\in[1,\infty)$, applying this, Lemma \ref{lem:strong-pointwise},
and Minkowski's inequality, for any $r\in\mathbb Z_+$ we obtain
\begin{align*}
\left\|A_r^\circ \right\|_{L^p(w)\to L^p(w)}\lesssim
\sum_{\genfrac{}{}{0pt}{}{j,m\in\mathbb Z_+}{j+m=r}}
\left\| T_{j,m}\right\|_{L^p(w)\to L^p(w)}
\lesssim [w]_{\mathscr{A}_{p,\varepsilon}}^{\frac{1}{p}}(r+1)
\frac{1}{k^{\frac{\varepsilon r}{p}}}.
\end{align*}
If $p\in(0,1)$, by inequality \eqref{eq-embeding-l1}
and Tonelli's theorem, we conclude that, for any $r\in\mathbb Z_+$
and $f\in L^p(w)$,
\begin{align*}
\|A_r^\circ f\|_{L^p(w)}^p
\lesssim[w]_{\mathscr{A}_{p,\varepsilon}}
(r+1)\frac{1}{k^{\varepsilon r}}\|f\|_{L^p(w)}^p.
\end{align*}
This completes the proof of Lemma \ref{lem:exp-class-synthesis}.
\end{proof}

We now give the proof of Theorem \ref{thm-Bound-Ap}.

\begin{proof}[Proof of Theorem \ref{thm-Bound-Ap}]
We first prove (I). To show the
implication (i) $\Longrightarrow$ (ii),
assume  that $M$ is bounded on $L^p(w)$.
Suppose that $f\in L^p(w)$ and, for any $N\in \mathbb{Z}_+$, let
$f_N:=f\mathbf1_{T_N}$. From the definitions of
operators $P_m$ and $M^\circ$, we deduce that, for
any $x\in T$ and $m\in\mathbb Z_+$,
\begin{align*}
P_mf(x)=P_mf_{|x|+m}(x)
\lesssim M^\circ f_{|x|+m}(x)
\end{align*}
and hence
\begin{align}\label{eq-sum-Pm}
\sum_{m\in\mathbb Z_+}\left|P_mf(x)\right|^p
\lesssim\sum_{N\in\mathbb Z_+}
\left[M^\circ f_N(x)\right]^p.
\end{align}
Moreover, by the definitions of
operators $Q_j$ and $M^\circ$, we find that,
for any $x\in T$ and $j\in[0,|x|]\cap\mathbb{Z}_+$,
\begin{align*}
Q_jf(x)=Q_jf_{|x|-j}(x)
\lesssim M^\circ f_{|x|-j}(x),
\end{align*}
and hence
\begin{align}\label{eq-sum-Qj}
\sum_{j\in\mathbb Z_+}\left|Q_jf(x)\right|^p\lesssim
\sum_{N\in\mathbb Z_+}\left[M^\circ f_N(x)\right]^p.
\end{align}
Since $M$ and $M^\circ$ are pointwise equivalent [see \eqref{eq-equivalent-M}],
using \eqref{eq-sum-Pm}, \eqref{eq-sum-Qj},
Tonelli's theorem, and the assumed boundedness of $M$ on $L^p(w)$,
we conclude that
\begin{align*}
&\sum_{m\in\mathbb Z_+}\|P_mf\|_{L^p(w)}^p
+\sum_{j\in\mathbb Z_+}\|Q_jf\|_{L^p(w)}^p\\
&\qquad\lesssim
\sum_{N\in\mathbb Z_+}\|M^\circ f_N\|_{L^p(w)}^p
\lesssim
\sum_{N\in\mathbb Z_+}\|f_N\|_{L^p(w)}^p
=\|f\|_{L^p(w)}^p.
\end{align*}
It then follows from Lemma \ref{lem-directional-decay} that
there exist $\theta_P,\theta_Q\in(0,1)$
such that, for any $m,j\in\mathbb{Z}_+$,
\begin{align*}
\|P_m\|_{L^p(w)\to L^p(w)}\lesssim\theta_P^m
\text{ and }
\|Q_j\|_{L^p(w)\to L^p(w)}\lesssim\theta_Q^j.
\end{align*}
Let $\rho:=\max\{\theta_P,\theta_Q\}$.
Applying the definition of $T_{j,m}$, we obtain
\begin{align*}
\|T_{j,m}\|_{L^p(w)\to L^p(w)}^p=
\|Q_j\circ P_m\|_{L^p(w)\to L^p(w)}^p
\lesssim\rho^{p(j+m)}.
\end{align*}
Choose $\varepsilon\in(0,\infty)$ sufficiently small
such that $k^\varepsilon\rho^p<1$. Then
\eqref{eq-characteristic-operator-form} yields
$w\in\mathscr{A}_{p,\varepsilon}$, and hence $w\in\mathscr{A}_p$.
This completes the proof of implication (i) $\Longrightarrow$ (ii).

To show the implication (ii) $\Longrightarrow$ (iii),
suppose that $w\in\mathscr{A}_{p,\varepsilon}$ for some
$\varepsilon\in(0,\infty)$.
By Lemma \ref{lem:exp-class-synthesis} and
the fact that $k\geq 2$, we find that,
for any $r\in\mathbb{Z}_+$,
\begin{align*}
\left\|A_r^\circ\right\|_{L^p(w)\to L^p(w)}
&\lesssim [w]^{\frac{1}{p}}_{\mathscr{A}_{p,\varepsilon}}
(r+1)^{\max\{1,\frac{1}{p}\}}
\frac{1}{k^{\frac{\varepsilon r}{p}}}\\
&=[w]^{\frac{1}{p}}_{\mathscr{A}_{p,\varepsilon}}
(r+1)^{\max\{1,\frac{1}{p}\}}
\frac{1}{k^{\frac{\varepsilon r}{2p}}}
\frac{1}{k^{\frac{\varepsilon r}{2p}}}\\
&\leq[w]^{\frac{1}{p}}_{\mathscr{A}_{p,\varepsilon}}
(r+1)^{\max\{1,\frac{1}{p}\}}
\frac{1}{2^{\frac{\varepsilon r}{2p}}}
\frac{1}{k^{\frac{\varepsilon r}{2p}}}
\lesssim[w]^{\frac{1}{p}}_{\mathscr{A}_{p,\varepsilon}}
\frac{1}{k^{\frac{\varepsilon r}{2p}}},
\end{align*}
where the last inequality follows from the fact that the sequence
$\{(r+1)^{\max\{1,\frac{1}{p}\}}
\frac{1}{2^{\frac{\varepsilon r}{2p}}}\}_{r\in\mathbb{Z}_+}$
is uniformly bounded on $r$. Thus,
\eqref{eq-Ar-exponential} holds.
This completes the proof of
implication (ii) $\Longrightarrow$ (iii).

Finally, we prove the implication (iii) $\Longrightarrow$ (i).
For this purpose, assume that \eqref{eq-Ar-exponential} holds for
some $\varepsilon\in(0,\infty)$. Observe that,
for any $f\in L^p(w)$ and $x\in T$,
\begin{align}\label{eq-M-supAr}
\left[M^\circ f(x)\right]^p\leq
\sum_{r\in\mathbb Z_+}\left[A_r^\circ f(x)\right]^p.
\end{align}
Since $M$ and $M^\circ$ are pointwise equivalent,
using the observation \eqref{eq-M-supAr}, Tonelli's theorem, and
assumption (iii), we conclude that,
for any $f\in L^p(w)$,
\begin{align}\label{eq-Mf-f}
\|Mf\|_{L^p(w)}^p\leq\sum_{r\in\mathbb Z_+}
\left\|A_r^\circ f\right\|_{L^p(w)}^p
\lesssim\sum_{r\in\mathbb Z_+}\frac{1}{k^{\varepsilon r p}}
\|f\|_{L^p(w)}^p\lesssim\|f\|_{L^p(w)}^p,
\end{align}
which further implies that $M$ is bounded on $L^p(w)$.
This completes the proof of the implication
(iii) $\Longrightarrow$ (i), and hence (I).

Finally, we show (II).
Note that the proof of the implication
(i) $\Longrightarrow$ (ii) above is also valid for $p\in(0,\frac{1}{2}]$.
If there exists a weight $w$ on $T$ such that $M$
is bounded on $L^p(w)$ for some $p\in(0,\frac{1}{2}]$,
then $w\in \mathscr{A}_p$. However,
by Proposition \ref{prop:nonempty-exp-class}, we find that
$\mathscr{A}_p=\emptyset$, which contradicts
the existence of such a weight $w$. This completes the proof of
(II), and hence Theorem \ref{thm-Bound-Ap}.
\end{proof}

Using Theorem \ref{thm-Bound-Ap}, we obtain the following
summability characterizations of
the boundedness of $M$ on $L^p(w)$.

\begin{corollary}\label{cor-sum-characterization}
Let $w$ be a weight on $T$. Then the following two statements hold.	
\begin{itemize}
\item[{\rm (i)}] If $p\in(1,\infty)$, then
$M$ is bounded on $L^p(w)$ if and only if
\begin{align}\label{eq-sum-characterization-superone}
\sum_{r\in\mathbb Z_+}\left(
\sum_{\genfrac{}{}{0pt}{}
{j,m\in\mathbb Z_+}{j+m=r}}
\left\{\frac{1}{k^{p(j+m)}}\sup_{x\in T}
w(D_j(x))\left[w^{-\frac{1}{p-1}}(D_m(x))
\right]^{p-1}\right\}^{\frac1p}\right)^{p}<\infty.
\end{align}

\item[{\rm (ii)}] If $p\in(\frac12,1]$,
then $M$ is bounded on $L^p(w)$ if and only if
\begin{align*}
\sum_{j,m\in\mathbb Z_+}\frac{1}{k^{p(j+m)}}
\sup_{x\in T}\left[w(D_j(x))
\max_{y\in D_m(x)}w^{-1}(y)\right]<\infty.
\end{align*}
\end{itemize}
\end{corollary}

\begin{proof}
We first prove (i). To this end, by Lemma \ref{lem-exact-two-layer-norm}, we find that
\eqref{eq-sum-characterization-superone} is equivalent to
\begin{align}\label{eq-sum-T-superone}
\Gamma:=\sum_{r\in\mathbb Z_+}\left[
\sum_{\genfrac{}{}{0pt}{}
{j,m\in\mathbb Z_+}{j+m=r}}
\|T_{j,m}\|_{L^p(w)\to L^p(w)}
\right]^p<\infty.
\end{align}
To show the necessity of (i),
suppose that $M$ is bounded on $L^p(w)$.
From Theorem \ref{thm-Bound-Ap}, we infer that there exists some
$\varepsilon\in(0,\infty)$ such that $w\in\mathscr A_{p,\varepsilon}$,
which, together with Lemma \ref{lem-exact-two-layer-norm},
further implies that, for any $j,m\in\mathbb{Z}_+$,
\begin{align*}
\|T_{j,m}\|_{L^p(w)\to L^p(w)}^p
\lesssim k^{-\varepsilon(j+m)}.
\end{align*}
Inserting this estimate into \eqref{eq-sum-T-superone},
we conclude that
\begin{align*}
\Gamma\lesssim\sum_{r\in\mathbb Z_+}(r+1)^p
k^{-\varepsilon r}<\infty
\end{align*}
and hence \eqref{eq-sum-T-superone} holds.
This completes the proof of the necessity of (i).
We next prove the sufficiency of (i).
For this purpose, we assume that \eqref{eq-sum-T-superone} holds.
By Lemma \ref{lem:strong-pointwise} and
Minkowski's inequality, we find that, for any
$r\in\mathbb{Z}_+$ and $f\in L^p(w)$,
\begin{align*}
\left\|A_r^\circ f\right\|_{L^p(w)}
\lesssim\sum_{\genfrac{}{}{0pt}{}
{j,m\in\mathbb Z_+}{j+m=r}}
\|T_{j,m}\|_{L^p(w)\to L^p(w)}
\|f\|_{L^p(w)}.
\end{align*}
Using \eqref{eq-equivalent-M},
\eqref{eq-M-supAr}, and Tonelli's theorem, we conclude that,
for any $f\in L^p(w)$,
\begin{align*}
\|Mf\|_{L^p(w)}^p\le\|M^\circ f\|_{L^p(w)}^p
\le\sum_{r\in\mathbb Z_+}\|A_r^\circ f\|_{L^p(w)}^p
\lesssim\|f\|_{L^p(w)}^p,
\end{align*}
where the last inequality follows from
\eqref{eq-sum-T-superone}. Thus,
$M$ is bounded on $L^p(w)$, which completes the proof of (i).

To show (ii), it suffices to repeat the above argument
with some slight modifications such as
Minkowski's inequality replaced by \eqref{eq-embeding-l1}; we omit the details.
This completes the proof of Corollary \ref{cor-sum-characterization}.
\end{proof}

As a direct corollary of Theorem \ref{thm-Bound-Ap}
and Proposition \ref{prop:exp-class-structure},
we have the following result.

\begin{corollary}
Let $p\in(\frac{1}{2},\infty)$ and $w$ be a weight on $T$.
If $M$ is bounded on $L^p(w)$, then there exists $\delta\in(\frac{1}{2},p)$
such that, for any $q\in[\delta,\infty)$,
$M$ is bounded on $L^q(w)$.
\end{corollary}

Finally, we prove Corollary \ref{cor-exp-quantitative}.

\begin{proof}[Proof of Corollary \ref{cor-exp-quantitative}]
Since $w\in \mathscr{A}_{p,\varepsilon}$,
from \eqref{eq-Mf-f} and Lemma \ref{lem:exp-class-synthesis},
we deduce that, for any $f\in L^p(w)$,
\begin{align}\label{eq-Mf-quan}
\left\|M f\right\|_{L^p(w)}^p
&\leq\sum_{r\in\mathbb Z_+}\left\|A_r^\circ f\right\|_{L^p(w)}^p\nonumber\\
&\lesssim [w]_{\mathscr{A}_{p,\varepsilon}}\sum_{r\in\mathbb Z_+}
(r+1)^{\max\{1,p\}}\frac{1}{k^{\varepsilon r}}\|f\|_{L^p(w)}^p\nonumber\\
&\lesssim[w]_{\mathscr{A}_{p,\varepsilon}}
\left(1-\frac{1}{k^{\varepsilon}}
\right)^{-\max\{1,p\}-1}\|f\|_{L^p(w)}^p,
\end{align}
where, in the last inequality, we used the basic inequality
$\sum_{r\in\mathbb{Z}_+}(r+1)^a q^r
\lesssim \frac{1}{(1-q)^{a+1}}$ for any $a\in(0,\infty)$ and $q\in(0,1)$ with the implicit positive
constant depending only on $a$. Note that $k\geq 2$.
Taking the $p$-th root of both sides of \eqref{eq-Mf-quan},
we obtain \eqref{eq-exp-quantitative}.

Finally, we prove the sharpness of the exponent $\frac{1}{p}$
of the weight constant. To this end, let
\begin{align*}
\beta:=\frac{\max\{p,1\}}2-1
\end{align*}
and $w_\beta$ be as in \eqref{eq-def-radial-beta}.
It follows from \eqref{eq-radial-constant} and  \eqref{eq-rangeradial-p>1} that
$w_\beta\in\mathscr{A}_{p,\varepsilon}$ and
$[w_\beta]_{\mathscr{A}_{p,\varepsilon}}=1$.
Let $\delta\in (0,1)$. Define the weight
$u_\delta$ on $T$ by setting
$u_\delta(o):=\delta$ and, for any $x\in T\setminus\{o\}$,
$u_\delta(x):=w_\beta(x)$.
Since the only descendant layer containing $o$ is $D_0(o)$,
by the definition of $u_\delta$ and \eqref{eq-radial-constant},
we find that $[u_\delta]_{\mathscr{A}_{p,\varepsilon}}
\sim\frac{1}{\delta}$, where the positive equivalence constants
are independent of $\delta$.
Let $f:=\mathbf1_{\{o\}}$. Observe that,
for any $x\in D_1(o)$, $o\in B(x,1)$ and hence
$Mf(x)\geq \frac{1}{k+2}$. Moreover, by the definition of
$u_\delta$, we find that, for any $x\in D_1(o)$,
$u_\delta(x)=w_\beta(x)=k^\beta$.
Using these observations, we conclude that
\begin{align}\label{eq-2}
\|M\|_{L^p(u_\delta)\to L^p(u_\delta)}
\geq\frac{\|Mf\|_{L^p(u_\delta)}}{\|f\|_{L^p(u_\delta)}}
\gtrsim\frac{1}{\delta^{\frac{1}{p}}},
\end{align}
where the implicit positive constant is independent of $\delta$.
If the exponent $\frac{1}{p}$ in \eqref{eq-exp-quantitative}
can be improved to a strictly smaller one, then
letting $\delta\to0$ in \eqref{eq-2} will lead to a contradiction.
Therefore, the exponent $\frac{1}{p}$ in \eqref{eq-exp-quantitative}
is optimal. This completes the proof of Corollary \ref{cor-exp-quantitative}.
\end{proof}

\subsection{Proofs of Theorems \ref{thm-strong} and \ref{thm-atomic-characterization}}\label{sub3-2}

In this subsection, we aim to prove
Theorems \ref{thm-strong} and \ref{thm-atomic-characterization}.
To this end, in the next lemma we show that the global Sawyer-type testing
condition \eqref{eq-global-ball-testing}
guarantees exponential decay of the operator norms
of $\{P_m\}_{m\in\mathbb{Z}_+}$
and $\{Q_m\}_{m\in\mathbb{Z}_+}$ on $L^p(w)$,
which is the key point used in the proof of
Theorem \ref{thm-strong}.

\begin{lemma}\label{lem-global-testing-decay}
Let $p\in(1,\infty)$, and let $w$ be a weight on $T$.
If \eqref{eq-global-ball-testing} holds,
then there exist positive
constants $C$ and $\theta_P,\theta_Q\in(0,1)$ such that,
for any $m\in\mathbb Z_+$ and $f\in L^p(w)$,
\begin{align}\label{eq-Pm-exp}
\|P_mf\|_{L^p(w)}
\leq C\theta_P^m\|f\|_{L^p(w)}
\end{align}
and
\begin{align}\label{eq-Qm-exp}
\|Q_mf\|_{L^p(w)}
\leq C\theta_Q^m\|f\|_{L^p(w)}.
\end{align}
Moreover, there exists a positive constant $\varepsilon$ such
that \eqref{eq-Ar-exponential} holds.
\end{lemma}

\begin{proof}
Let
\begin{align}\label{eq-lambda}
\lambda:=\sup_{B\subset T}\frac{\|M(\sigma\mathbf{1}_B)
\|_{L^p(w)}^p}{\sigma(B)}\in[1,\infty),
\end{align}
where the supremum is taken over all balls $B\subset T$
and $\sigma$ is as in \eqref{eq-def-dualweight}.
We first prove \eqref{eq-Pm-exp}.
For any $m\in\mathbb Z_+$ and $x\in T$, let
\begin{align}\label{eq-def-alpham}
\alpha_m(x):=
\frac{w(x)[\sigma(D_m(x))]^{p-1}}{k^{pm}}
\end{align}
and
\begin{align*}
\mathscr{P}_m:=\sup_{x\in T}\alpha_m(x).
\end{align*}
By Lemma \ref{lem-exact-two-layer-norm} with $j=0$,
we find that, for any $m\in\mathbb{Z}_+$,
\begin{align}\label{eq-Pm-exact-norm}
\|P_m\|_{L^p(w)\to L^p(w)}^p=\mathscr{P}_m.
\end{align}

We next show that $\mathscr{P}_m\to 0$ as $m\to\infty$.
To this end, we first establish a uniform estimate for $\alpha_m(x)$.
Fix $m\in\mathbb Z_+$ and $x\in T$ and let $B_m:=B(x,m)$.
Taking the ball $B$ in the definition \eqref{eq-def-M}
of $M$ to be $B_m$, we obtain
$$
M(\sigma\mathbf{1}_{B_m})(x)
\geq\frac{\sigma(B_m)}{|B_m|}.
$$
It follows from \eqref{eq-global-ball-testing},
\eqref{eq-sphere-ball-growth}, and \eqref{eq-lambda} that
\begin{align*}
\frac{w(x)[\sigma(B_m)]^{p-1}}{k^{pm}}
\lesssim\frac{\|M(\sigma\mathbf{1}_{B_m})
\|_{L^p(w)}^p}{\sigma(B_m)}\leq\lambda.
\end{align*}
Using $D_m(x)\subset B_m$, we conclude that
\begin{align}\label{eq-alpha-uniform}
\alpha_m(x)\leq\frac{w(x)[\sigma(B_m)]^{p-1}}{k^{pm}}
\lesssim\lambda,
\end{align}
which is the desired uniform estimate for $\alpha_m(x)$.
Let $R:=D_m(x)$.
For any $j\in\{0,\ldots,m\}$, let $n:=m-j$ and
$g_j:=P_n(\sigma\mathbf{1}_R)$.
If $y\in D_j(x)$, then $D_n(y)\subset R\subset B_m$.
By the definitions of $M$ and $P_n$ and by \eqref{eq-sphere-ball-growth},
we find that, for any $y\in D_j(x)$,
$$
g_j(y)=\frac{\sigma(D_n(y))}{k^n}
\lesssim\frac{\sigma(D_n(y))}{|B(y,n)|}
\leq M(\sigma\mathbf{1}_{B_m})(y).
$$
Applying the fact that the sets $\{D_0(x),\ldots,D_m(x)\}$ are pairwise disjoint
and applying the testing condition \eqref{eq-global-ball-testing} yield
\begin{align}\label{eq-energy-upper}
\sum_{j=0}^m\left\|g_j\mathbf{1}_{D_j(x)}\right\|^p_{L^p(w)}
&\lesssim\sum_{j=0}^m\sum_{y\in D_j(x)}
\left[M(\sigma\mathbf{1}_{B_m})(y)\right]^p w(y)\nonumber\\
&\leq\|M(\sigma\mathbf{1}_{B_m})\|_{L^p(w)}^p
\leq\lambda\sigma(B_m).
\end{align}
From the fact that the sets
$\{D_n(y):\,y\in D_j(x)\}$ form a disjoint partition of $R$
and that $j+n=m$, we infer that
$$P_jg_j(x)=P_jP_n(\sigma\mathbf{1}_R)(x)
=P_m(\sigma\mathbf{1}_R)(x)
=\frac{\sigma(R)}{k^m}$$
and hence
$$
w(x)|P_jg_j(x)|^p=
w(x)\left[\frac{\sigma(R)}{k^m}\right]^p
=\alpha_m(x)\sigma(R).
$$
This, together with H\"{o}lder's inequality, further implies that
\begin{align*}
\alpha_m(x)\sigma(R)=w(x)|P_jg_j(x)|^p\leq\alpha_j(x)
\sum_{y\in D_j(x)}|g_j(y)|^p w(y)
=\alpha_j(x)\left\|g_j\mathbf{1}_{D_j(x)}
\right\|^p_{L^p(w)}.
\end{align*}
Note that it was proved in \eqref{eq-alpha-uniform} that
$\alpha_j(x)$ is uniformly bounded by $\lambda$.
Inserting this bound into the last inequality and summing
$j$ from $0$ to $m$, we obtain
\begin{align}\label{eq-energy-lower}
(m+1)\frac{\alpha_m(x)}{\lambda}\sigma(R)
\lesssim\sum_{j=0}^m \left\|g_j
\mathbf{1}_{D_j(x)}\right\|^p_{L^p(w)}
\lesssim\lambda\sigma(B_m),
\end{align}
where, in the last step, we used \eqref{eq-energy-upper}.
By $\sigma(R)\neq0$, combining \eqref{eq-alpha-uniform},
\eqref{eq-energy-lower}, and the definition
\eqref{eq-def-alpham} of $\alpha_m(x)$, we conclude that
$$
\frac{m+1}{\lambda}\alpha_m(x)\lesssim\lambda
\left[\frac{\lambda}{\alpha_m(x)}\right]^{\frac{1}{p-1}},
$$
which is equivalent to
$$
(m+1)[\alpha_m(x)]^{p'}\lesssim\lambda^{p'+1}.
$$
Consequently,
\begin{align}\label{eq-decay-Am}
\mathscr{P}_m\lesssim\lambda^{1+\frac{1}{p'}}
\frac{1}{(m+1)^{\frac{1}{p'}}}.
\end{align}
Note that $\mathscr{P}_m\to0$ as $m\to\infty$.
From \eqref{eq-def-Pm}, we deuce that, for any
$\ell,m\in\mathbb{Z}_+$ and $f\in L^p(w)$,
$P_\ell(P_mf)=P_{\ell+m}f$,
which, combined with \eqref{eq-Pm-exact-norm},
further implies that $\mathscr{P}_{\ell+m}
\leq\mathscr{P}_\ell\mathscr{P}_m$.
By \eqref{eq-decay-Am}, we can
choose $m_0\in\mathbb N$ such that
$\mathscr{P}_{m_0}\in(0,1)$.
Observe that, for any $m\in\mathbb{Z}_+$,
$m=qm_0+r$, where $q\in\mathbb{Z}_+$ and
$r\in[0, m_0)\cap\mathbb{Z}_+$.
It immediately follows that
\begin{align*}
\left\|P_m\right\|_{L^p(w)\to L^p(w)}
\leq\left\|P_{m_0}\right\|_{L^p(w)\to L^p(w)}^q
\left\|P_r\right\|_{L^p(w)\to L^p(w)}\lesssim
\mathscr{P}^\frac{m}{p m_0}_{m_0}.
\end{align*}
Thus, \eqref{eq-Pm-exp} holds with
$\theta_P:=\mathscr{P}^\frac{1}{p m_0}_{m_0}
\in(0,1)$. This completes the proof of \eqref{eq-Pm-exp}.

We next prove \eqref{eq-Qm-exp}.  Fix $z\in T$.
Observe that, for any $m\in\mathbb{Z}_+$ and $x\in D_m(z)$,
$z\in B(x,m)$ and hence
$$
M(\sigma\mathbf{1}_{\{z\}})(x)
\gtrsim \frac{1}{k^m}\sigma(z).
$$
It follows from the disjointness of the sets
$\{D_m(z)\}_{m\in\mathbb{Z}_+}$ that
\begin{align*}
\sigma(z)^p\sum_{m\in\mathbb Z_+}\frac{1}{k^{pm}}w(D_m(z))
\lesssim\|M(\sigma\mathbf{1}_{\{z\}})\|_{L^p(w)}^p
\leq\lambda \sigma(z).
\end{align*}
Since $[\sigma(z)]^{1-p}=w(z)$, we deduce that
\begin{align*}
\sum_{m\in\mathbb Z_+}\frac{1}{k^{pm}}w(D_m(z))
\lesssim\lambda w(z).
\end{align*}
By Tonelli's theorem and \eqref{eq-def-Qm},
we find that, for any $f\in L^p(w)$,
\begin{align*}
\sum_{m\in\mathbb Z_+}\|Q_mf\|_{L^p(w)}^p
=\sum_{z\in T}|f(z)|^p
\sum_{m\in\mathbb Z_+}\frac{1}{k^{pm}}w(D_m(z))
\lesssim\lambda\|f\|_{L^p(w)}^p.
\end{align*}
From Lemma \ref{lem-directional-decay}, we infer that
\eqref{eq-Qm-exp} holds for some $\theta_Q\in(0,1)$.
This completes the proof of \eqref{eq-Qm-exp}.

Finally, let $\rho:=\max\{\theta_P,\theta_Q\}\in(0,1)$.
Using Lemma \ref{lem:strong-pointwise}, \eqref{eq-Pm-exp}, and
\eqref{eq-Qm-exp}, we conclude that, for any $r\in\mathbb{Z}_+$
and $f\in L^p(w)$,
\begin{align}\label{eq-Ar=PmQj}
\left\|A_r^\circ f\right\|_{L^p(w)}
&\lesssim\sum_{\genfrac{}{}{0pt}{}{j,m\in\mathbb Z_+}{j+m=r}}
\left\|Q_j(P_mf)\right\|_{L^p(w)}\nonumber\\
&\lesssim\sum_{\genfrac{}{}{0pt}{}{j,m\in\mathbb Z_+}{j+m=r}}
\theta_Q^j\theta_P^m\|f\|_{L^p(w)}
\leq (r+1)\rho^r\|f\|_{L^p(w)}.
\end{align}
Choose $\widetilde\rho\in(\rho,1)$.  Since
$(r+1)\rho^r\lesssim\widetilde\rho^r$,
by writing $\widetilde\rho=\frac{1}{k^{\varepsilon}}$
for some $\varepsilon\in(0,\infty)$,
we obtain \eqref{eq-Ar-exponential}.
This completes the proof of Lemma
\ref{lem-global-testing-decay}.
\end{proof}

We now give the proof of Theorem \ref{thm-strong}.

\begin{proof}[Proof of Theorem \ref{thm-strong}]
By Theorem \ref{thm-Bound-Ap}, we find that, to
prove this theorem, it suffices to
show that (i), (ii), and (iii) of this theorem are mutually
equivalent. For this purpose, we first prove  the implication
(i) $\Longrightarrow$ (ii).
Suppose that (i) holds.  For any ball $B\subset T$,
it follows immediately from the assumed boundedness of $M$ that
\begin{align*}
\left\|M\left(\sigma\mathbf{1}_B\right)\right\|_{L^p(w)}^p
&\lesssim\left\|\sigma\mathbf{1}_B\right\|_{L^p(w)}^p
=\sum_{x\in B}[\sigma(x)]^pw(x)=\sigma(B)
\end{align*}
and hence (ii) holds. This completes the proof of  the implication
(i) $\Longrightarrow$ (ii).

To show  the implication (ii) $\Longrightarrow$ (iii),
assume that (ii) holds. By Lemma
\ref{lem-global-testing-decay}, there exists $\varepsilon\in(0,\infty)$
such that \eqref{eq-Ar-exponential} holds.
Since $k\geq 2$,
we may assume that $\varepsilon\in(0,1)$. Note that,
for any $E,F\subset T$ and $r\in\mathbb Z_+$,
\begin{align*}
(\mathbf{1}\otimes w)\left(I_r(E,F)\right)
&=\sum_{y\in F}w(y)\sum_{\genfrac{}{}{0pt}{}{x\in E}{d(x,y)=r}}1
=\sum_{y\in F}|S(y,r)|A_r^\circ(\mathbf{1}_E)(y)w(y)\\
&\sim k^r\sum_{y\in T}A_r^\circ(\mathbf{1}_E)(y)\mathbf{1}_F(y)w(y),
\end{align*}
where $I_r(E,F)$ is as in \eqref{eq-r-incidence-set}.
From H\"older's inequality and \eqref{eq-Ar-exponential},
we deduce that
\begin{align*}
(\mathbf{1}\otimes w)\left(I_r(E,F)\right)
&\lesssim k^r\left\|A_r^\circ(\mathbf{1}_E)
\right\|_{L^p(w)}\|\mathbf{1}_F\|_{L^{p'}(w)}\\
&\lesssim k^{(1-\varepsilon)r}\left[w(E)\right]^{\frac{1}{p}}
\left[w(F)\right]^{1-\frac{1}{p}}.
\end{align*}
Thus, \eqref{eq-MA-alpha-one} holds with $\beta:=1-\varepsilon
\in(0,1)$, which completes the proof of the implication
(ii) $\Longrightarrow$ (iii).

If \eqref{eq-MA-alpha-one} holds,
applying \cite[Theorem 1.1]{or22} with $\alpha:=1$, we conclude that
$M$ is bounded on $L^p(w)$. This completes the proof of
the implication (iii) $\Longrightarrow$ (i), and hence Theorem \ref{thm-strong}.
\end{proof}

Finally, we prove Theorem \ref{thm-atomic-characterization}.

\begin{proof}[Proof of Theorem \ref{thm-atomic-characterization}]
By Theorem \ref{thm-Bound-Ap}, we find that, to
prove this theorem, it suffices to
show that $M$ is bounded on $L^p(w)$ if and only if
\eqref{eq-test-onepoint} holds. To this end,
we first prove the necessity. Using \eqref{eq-sphere-ball-growth},
we conclude that, for any $x,y\in T$,
\begin{align}\label{eq-test-M-onepoint}
M(\mathbf1_{\{y\}})(x)\sim\frac{1}{k^{d(x,y)}} .
\end{align}
It immediately follows from the boundedness of
$M$ on $L^p(w)$ that, for any $y\in T$,
\begin{align*}
\sum_{x\in T}\frac{w(x)}{k^{p d(x,y)}}\sim
\left\|M(\mathbf1_{\{y\}})\right\|_{L^p(w)}^p
\lesssim\left\|\mathbf1_{\{y\}}\right\|_{L^p(w)}^p=w(y),
\end{align*}
which further implies that \eqref{eq-test-onepoint} holds.
This completes the proof of the necessity.

To show the sufficiency, by \eqref{eq-sphere-ball-growth},
we find that, for any $f\in L^p(w)$ and $x\in T$,
\begin{align}\label{eq-M-kernel}
Mf(x)=\sup_{r\in\mathbb{Z}_+}\frac{1}{|B(x,r)|}
\sum_{y\in B(x,r)}\left|f(y)\right|
\lesssim\sum_{y\in T}\frac{1}{k^{d(x,y)}}\left|f(y)\right|.
\end{align}
Using this, \eqref{eq-embeding-l1}, \eqref{eq-test-onepoint},
and Tonelli's theorem, we conclude that, for any $f\in L^p(w)$,
\begin{align*}
\|Mf\|_{L^p(w)}^p&=\sum_{x\in T}|Mf(x)|^p w(x)
\lesssim\sum_{x\in T}\sum_{y\in T}\frac{1}{k^{pd(x,y)}}\left|f(y)\right|^p w(x)\\
&=\sum_{y\in T}\sum_{x\in T}\frac{w(x)}{k^{pd(x,y)}}\left|f(y)\right|^p
\lesssim\sum_{y\in T}|f(y)|^pw(y)=\|f\|_{L^p(w)}^p.
\end{align*}
Thus, $M$ is bounded on $L^p(w)$. This completes the proof of the sufficiency,
and hence Theorem \ref{thm-atomic-characterization}.
\end{proof}

\section{Radial Weights and Their Characterizations}\label{sec:radial}

In this section, for any $p\in(\frac{1}{2}, \infty)$,
we establish the characterizations of the
boundedness of $M$ on $L^p(w)$ and also from $L^p(w)$ to $L^{p,\infty}(w)$
for radial weights $w$.
Moreover, when $p\in(0,\frac{1}{2}]$, we also prove that
there exists no weight $w$ on $T$ such that
$M$ is bounded from $L^p(w)$ to $L^{p,\infty}(w)$.
Furthermore, for the exponential radial weights $w_\beta$
in \eqref{eq-def-radial-beta} we determine the exact ranges
of $\beta$ such that $M$ is bounded on $L^p(w_\beta)$ and
from $L^p(w_\beta)$ to $L^{p,\infty}(w_\beta)$.

In \cite{or22}, Ombrosi and Rivera-R\'{\i}os found some examples
of radial weights to illustrate that the weighted boundedness properties
of $M$ on $T$ differ from their Euclidean space counterparts.
Here, a weight $w$ on $T$ is said to be \emph{radial} if
there exists a sequence $\{a_N\}_{N\in\mathbb{Z}_+}$ in
$(0,\infty)$ such that
$$
w=\sum_{N\in \mathbb{Z}_+}a_N \mathbf{1}_{T_N},
$$
where $T_N$ is as in \eqref{eq-def-TN}.
Suppose that $w:=\sum_{N\in \mathbb{Z}_+}a_N \mathbf{1}_{T_N}$
is a radial weight on $T$. For any $N\in\mathbb{Z}_+$,
let $W_N:=w(T_N)=k^N a_N$.
Recall that, for any weight $w$ on $T$
and any $p\in(0,\infty)$,
the \emph{weighted weak Lebesgue space} $L^{p,\infty}(w)$
on $T$ is defined to be the set
of all functions $f$ on $T$ such that
$$
\|f\|_{L^{p,\infty}(w)}:=\sup_{\lambda\in(0,\infty)}\lambda\,
\left[w\left(\{x\in T:\,|f(x)|>\lambda\}\right)
\right]^{\frac{1}{p}}<\infty.
$$
For radial weights, the following theorem gives the
characterizations of the weighted strong-type and
weak-type boundedness of $M$ in terms of
explicit conditions on $W_N$.
\begin{theorem}\label{thm-bound-radial}
Let $w$ be a radial weight.
\begin{itemize}
\item[{\rm (i)}] If $p\in(\frac{1}{2},\infty)$, then $M$ is bounded on
$L^p(w)$ if and only if
\begin{align}\label{eq-radial-1}
\sup_{N\in\mathbb Z_+}\frac{1}{W_N}
\sum_{j=0}^{N}k^{(1-p)_+j}W_{N-j}<\infty
\end{align}
and
\begin{align}\label{eq-radial-2}
\sup_{N\in\mathbb Z_+}\frac{1}{W_N}
\sum_{m\in\mathbb Z_+}\frac{1}{k^{pm}}W_{N+m}<\infty.
\end{align}

\item[{\rm (ii)}] If $p\in(\frac{1}{2},1)$, then $M$ is bounded from
$L^p(w)$ to $L^{p,\infty}(w)$ if and only if
\begin{align}\label{eq-radial-weak-subone}
\sup_{N,t\in\mathbb Z_+}\frac{1}{k^{pt}W_N}
\sum_{j=0}^{\min\{N,t\}}k^j\sum_{m=0}^{t-j}W_{N-j+m}<\infty.
\end{align}

\item[{\rm (iii)}] If $p\in[1,\infty)$, then $M$ is bounded from
$L^p(w)$ to $L^{p,\infty}(w)$ if and only if
\eqref{eq-radial-1} holds and
\begin{align}\label{eq-radial-3}
\sup_{N,t\in\mathbb Z_+}
\frac{1}{W_Nk^{pt}}\sum_{m=0}^{t}W_{N+m}<\infty.
\end{align}
\end{itemize}
\end{theorem}

The following proposition shows that, for any $p\in(0,\frac{1}{2}]$,
there exists no weight $w$ on $T$ such that
$M$ is bounded from $L^p(w)$ to $L^{p,\infty}(w)$.
Thus, the range $p\in(\frac{1}{2},\infty)$
considered in items (ii) and (iii)
of Theorem \ref{thm-bound-radial} is sharp.

\begin{proposition}\label{prop:weak-lower-endpoint}
Let $p\in(0,\frac{1}{2}]$. Then there exists no weight $w$ on $T$ such that
$M$ is bounded from $L^p(w)$ to $L^{p,\infty}(w)$.
\end{proposition}

\begin{proof}
Suppose, to the contrary, that such a weight $w$ exists.
Applying \eqref{eq-test-M-onepoint} and the assumed
weak-type boundedness of $M$, we conclude that,
for any $y\in T$ and $r\in\mathbb{Z}_+$,
\begin{align}\label{eq-general-weak-sphere-growth}
w(S(y,r))\lesssim k^{pr}w(y).
\end{align}
For any $r\in\mathbb{Z}_+$, $f\in L^2(T)$, and $x\in T$, let
\begin{align*}
\mathcal{S}_r f(x):=\sum_{y\in S(x,r)}f(y).
\end{align*}
We next claim that, for any $r\in\mathbb{Z}_+$,
\begin{align}\label{eq-claim-spherical}
\left\|\mathcal S_r\right\|_{L^2(T)\to L^2(T)}
\lesssim k^{pr}.
\end{align}
This inequality can be proved by the argument used in
the proof of \cite[Theorem 19]{rt91}.
But, for completeness, we give an elementary proof here.
By the Cauchy--Schwarz inequality and \eqref{eq-general-weak-sphere-growth},
we find that, for any $f\in L^2(T)$ and $x\in T$,
\begin{align*}
\left|\mathcal S_r f(x)\right|^2\leq\left[\sum_{y\in S(x,r)}w(y)\right]
\left[\sum_{y\in S(x,r)}\frac{|f(y)|^2}{w(y)}\right]\lesssim
k^{pr}w(x)\sum_{y\in S(x,r)}\frac{|f(y)|^2}{w(y)}.
\end{align*}
Summing over $x\in T$ and using Tonelli's theorem and
\eqref{eq-general-weak-sphere-growth} again, we conclude that
\begin{align*}
\left\|\mathcal S_r f\right\|_{L^2(T)}^2\lesssim k^{pr}
\sum_{y\in T}\frac{|f(y)|^2}{w(y)}
\sum_{x\in S(y,r)}w(x)
\lesssim k^{2pr}\|f\|_{L^2(T)}^2,
\end{align*}
and hence the claim \eqref{eq-claim-spherical} holds.
We next establish a lower bound for the
operator norms of $\{\mathcal{S}_r\}_{r\geq 2}$ on $L^2(T)$.
To this end, fix $r\in\mathbb{N}$ with $r\geq 2$ and,
for any $N\in\mathbb{N}$ satisfying $N\geq2r$ and $x\in T$, let
\begin{align*}
g_N(x):=\sum_{n=0}^{N}k^{-\frac{|x|}{2}}
\mathbf1_{T_n}(x).
\end{align*}
Since $|T_n|=k^n$ for any $n\in\mathbb{Z}_+$, it follows that
\begin{align}\label{eq-norm-gN}
\|g_N\|_{L^2(T)}^2=\sum_{n=0}^N 1=N+1.
\end{align}
Suppose that $x\in T$ satisfies $r\leq |x|\leq N-r$.
By the branch structure of $T$, we find that
$S(x,r)$ can be decomposed into the following
mutually disjoint sets
\begin{align}\label{eq-decomp-Sxr}
S(x,r)=\bigcup_{j=0}^{r}\left\{y\in S(x,r):\,x\wedge y=x^{(j)}\right\}
=:\bigcup_{j=0}^{r}E_j(x,r).
\end{align}
Observe that, for any $j\in\{1,\dots,r-1\}$,
$|E_j(x,r)|=(k-1)k^{r-j-1}$, $|E_0(x,r)|=k^r$, and $|E_r(x,r)|=1$.
From these, the decomposition in \eqref{eq-decomp-Sxr}, and the definition of
$g_N$, we deduce that
\begin{align*}
\mathcal S_r g_N(x)&=\sum_{j=0}^{r}
\sum_{y\in E_j(x,r)}g_N(y)
=\left[k^{\frac r2}+(k-1)\sum_{j=1}^{r-1}
k^{\frac{r}{2}-1}+k^{\frac r2}\right]g_N(x)\\
&=\left[2+\frac{k-1}{k}(r-1)\right]k^{\frac r2}g_N(x)
\geq\frac{r+3}{2}k^{\frac r2}g_N(x),
\end{align*}
where, in the last step, we used the assumption that $k\geq 2$.
Thus,
\begin{align*}
\|\mathcal S_r g_N\|_{L^2(T)}^2
\gtrsim r^2 k^r\sum_{n=r}^{N-r}\sum_{x\in T_n}|g_N(x)|^2
=r^2 k^r(N-2r+1),
\end{align*}
which, together with \eqref{eq-norm-gN}, further implies that
\begin{align*}
\|\mathcal S_r\|_{L^2(T)\to L^2(T)}
\gtrsim r k^{\frac r2}
\left(\frac{N-2r+1}{N+1}\right)^{\frac12}.
\end{align*}
Letting $N\to\infty$, we conclude that
\begin{align}\label{eq-spherical-sum-lower}
\|\mathcal S_r\|_{L^2(T)\to L^2(T)}\gtrsim r k^{\frac r2}.
\end{align}
Combining \eqref{eq-claim-spherical} and
\eqref{eq-spherical-sum-lower}, we find that,
for any $r\in\mathbb{N}$ with $r\geq 2$,
$r\lesssim k^{(p-\frac12)r}$,
which is impossible when $p\in(0,\frac{1}{2}]$.
This completes the proof of Proposition \ref{prop:weak-lower-endpoint}.
\end{proof}

For any $p\in(\frac{1}{2},1)$, the following proposition
gives a characterization of the boundedness of $M$
from $L^p(w)$ to $L^{p,\infty}(w)$.

\begin{proposition}\label{prop:weak-atomic-characterization}
Let $p\in(\frac{1}{2},1)$, and let $w$ be a weight on $T$. Then $M$ is bounded
from $L^p(w)$ to $L^{p,\infty}(w)$ if and only if
\begin{align}\label{eq-weak-atomic-ball}
\sup_{y\in T}\sup_{r\in\mathbb Z_+}
\frac{w(B(y,r))}{k^{pr}w(y)}<\infty.
\end{align}
\end{proposition}

\begin{proof}
To prove the necessity, suppose first that
$M$ is bounded from $L^p(w)$ to $L^{p,\infty}(w)$.
By \eqref{eq-test-M-onepoint}, we find that, for any
$y\in T$ and $x\in B(y,r)$,
$M(\mathbf1_{\{y\}})(x)\gtrsim \frac{1}{k^r}$.
Therefore, testing the boundedness of $M$ on the function
$\mathbf1_{\{y\}}$ yields \eqref{eq-weak-atomic-ball}.
This completes the proof of the necessity.

We next show the sufficiency. From \eqref{eq-test-M-onepoint}
and \eqref{eq-weak-atomic-ball}, we infer that, for any $y\in T$,
\begin{align*}
\|M(\mathbf1_{\{y\}})\|_{L^{p,\infty}(w)}^p
\lesssim\sup_{r\in\mathbb Z_+}
\frac{w(B(y,r))}{k^{pr}}\lesssim w(y),
\end{align*}
where the implicit positive constant is independent of $y$.
Using \eqref{eq-test-M-onepoint}, \eqref{eq-M-kernel},
and the $p$-convexity of $L^{p,\infty}(w)$
(see, for instance, \cite[Theorem 5]{arino87}),
we conclude that, for any $f\in L^{p}(w)$,
\begin{align*}
\left\|Mf\right\|_{L^{p,\infty}(w)}^p
&\lesssim\left\|\sum_{y\in T}|f(y)|M(\mathbf1_{\{y\}})
\right\|_{L^{p,\infty}(w)}^p
\lesssim\sum_{y\in T}|f(y)|^p
\left\|M(\mathbf1_{\{y\}})\right\|_{L^{p,\infty}(w)}^p\\
&\lesssim\sum_{y\in T}|f(y)|^pw(y)
=\|f\|_{L^p(w)}^p,
\end{align*}
which further implies that $M$ is bounded
from $L^p(w)$ to $L^{p,\infty}(w)$.
This completes the proof of the sufficiency,
and hence Proposition \ref{prop:weak-atomic-characterization}.
\end{proof}

\begin{remark}
It is also a natural question to
characterize the conditions on $w$ such that
$M$ is bounded from $L^p(w)$ to $L^{p,\infty}(w)$
when $p\in[1,\infty)$. However, the proof of Proposition \ref{prop:weak-atomic-characterization}
depends on the $p$-convexity of $L^{p,\infty}(w)$,
which only holds when $p\in(0,1)$.
Therefore, the proof of Proposition \ref{prop:weak-atomic-characterization}
does not extend directly to the case where $p\in[1,\infty)$.
In particular, when $p\in(1,\infty)$, it was proved in
\cite[Theorem 1.1]{or22} that, if there exists $\beta\in(0,1)$
such that, for any $r\in\mathbb{Z}_+$ and $E,F\subset T$,
\begin{align*}
(\mathbf{1}\otimes w)\left(I_r(E,F)\right)
\lesssim k^{\beta r}[w(E)]^{\frac{\beta}{p}}[w(F)]^{1-\frac{\beta}{p}},
\end{align*}
then $M$ is bounded from $L^p(w)$ to $L^{p,\infty}(w)$,
where $I_r(E,F)$ is as in \eqref{eq-r-incidence-set}.
It is still unknown whether this condition is necessary.
\end{remark}

The following lemma provides the necessary condition in
Theorem \ref{thm-bound-radial}.

\begin{lemma}\label{lem:radial-upward-necessary}
Let $p\in[1,\infty)$, and let $w$ be a radial weight on $T$.
If $M$ is bounded from $L^p(w)$ to $L^{p,\infty}(w)$, then
\eqref{eq-radial-1} holds.
\end{lemma}

\begin{proof}
Fix $N\in\mathbb Z_+$ and take $f:=\mathbf1_{T_N}$. If
$j\in[0,N]\cap\mathbb Z_+$ and $x\in T_j$, then
$D_{N-j}(x)\subset T_N$. By
\eqref{eq-sphere-ball-growth} and \eqref{eq-def-Mcicr},
we find that there exists a positive constant $c$ such that
$$
M^\circ f(x)\geq
\frac{|D_{N-j}(x)|}{|S(x,N-j)|}\geq c.
$$
It follows from the pointwise equivalence of $M$ and $M^\circ$
and from the assumed weak-type boundedness of $M$ that
$$
\sum_{j=0}^{N}W_j
\lesssim w\left(\left\{x\in T:\,M^\circ f(x)>
\frac{c}{2}\right\}\right)
\lesssim\|f\|_{L^p(w)}^p=W_N.
$$
Noting $(1-p)_+=0$ now, this inequality is exactly \eqref{eq-radial-1}.
This then completes the proof of Lemma \ref{lem:radial-upward-necessary}.
\end{proof}

In the next lemma, we show that \eqref{eq-radial-1}
guarantees the exponential decay of the
operator norms of $\{P_m\}_{m\in\mathbb{Z}_+}$ on $L^p(w)$,
which will be used in the proof of Theorem \ref{thm-bound-radial}.

\begin{lemma}\label{lem:radial-Pm-decay}
Let $p\in[1,\infty)$, and let $w$ be a radial weight on $T$.
If \eqref{eq-radial-1} holds, then there exist a positive constant
$C$ and $\theta\in(0,1)$ such that, for any $N\in\mathbb Z_+$ and
$m\in[0,N]\cap\mathbb Z_+$,
\begin{align}\label{eq-A-self-improve}
W_{N-m}\leq C\theta^mW_N.
\end{align}
Moreover, for any $m\in\mathbb Z_+$ and $f\in L^p(w)$,
\begin{align}\label{eq-Pm-radial-decay}
\|P_mf\|_{L^p(w)}^p
\leq C\theta^m\|f\|_{L^p(w)}^p.
\end{align}
\end{lemma}

\begin{proof}
We first prove \eqref{eq-A-self-improve}. To this end,
for any $N\in\mathbb Z_+$, let
$
S_N:=\sum_{j=0}^{N}W_j.
$
By \eqref{eq-radial-1}, there exists $C\in(1,\infty)$ such that,
for any $N\in\mathbb Z_+$, $S_N\leq C W_N$. Since
$S_N=S_{N-1}+W_N$ for any $N\in\mathbb N$, we infer that
$S_{N-1}\leq(1-\frac{1}{C})S_N$.
Let $\theta:=1-\frac{1}{C}\in(0,1)$.  Iterating this estimate,
for any $N\in\mathbb{Z}_+$ and $m\in[0,N]\cap\mathbb{Z}_+$
we obtain $S_{N-m}\leq\theta^mS_N$.
This, together with $W_{N-m}\leq S_{N-m}$ and $S_N\leq C W_N$,
further implies \eqref{eq-A-self-improve}.

To show \eqref{eq-Pm-radial-decay}, applying H\"{o}lder's inequality,
we conclude that, for any $m\in\mathbb Z_+$, $f\in L^p(w)$, and $x\in T$,
$$
|P_mf(x)|^p\leq \frac{1}{k^{m}}\sum_{y\in D_m(x)}|f(y)|^p.
$$
Since $w$ is radial, from \eqref{eq-A-self-improve}, we deduce that,
for any $m\in\mathbb Z_+$ and $f\in L^p(w)$,
\begin{align*}
\|P_mf\|_{L^p(w)}^p
&\leq \frac{1}{k^{m}}\sum_{x\in T}w(x)
\sum_{y\in D_m(x)}|f(y)|^p\\
&=\sum_{N\geq m}\sum_{y\in T_N}|f(y)|^p
\frac{W_{N-m}}{k^N}\\
&\lesssim\theta^m
\sum_{N\geq m}\sum_{y\in T_N}|f(y)|^p
\frac{W_N}{k^N}
\leq\theta^m\|f\|_{L^p(w)}^p,
\end{align*}
which completes the proof of \eqref{eq-Pm-radial-decay},
and hence Lemma \ref{lem:radial-Pm-decay}.
\end{proof}

To prove the weak-type estimate in Theorem \ref{thm-bound-radial},
we  define the operator $H^*$ by setting,
for any function $f$ on $T$ and for any $x\in T$,
\begin{align*}
H^*f(x):=\sup_{j\in[0,|x|]\cap\mathbb{Z}_+}\left|Q_jf(x)\right|
=\sup_{j\in[0,|x|]\cap\mathbb{Z}_+}\frac{1}{k^j}\left|f\left(x^{(j)}\right)\right|,
\end{align*}
where $Q_j$ is the same as in \eqref{eq-def-Qm}.
For any $p\in(0,\infty)$,
the following lemma establishes the characterization of
the boundedness of $H^*$ from $L^p(w)$ to $L^{p,\infty}(w)$.

\begin{lemma}\label{lem:Manc-weak}
Let $p\in(0,\infty)$, and let $w$ be a weight on $T$.
Then $H^*$ is bounded from $L^p(w)$ to $L^{p,\infty}(w)$
if and only if
\begin{align}\label{eq-Sh-Wp}
\sup_{z\in T}\sup_{t\in\mathbb Z_+}
\frac{1}{k^{pt}w(z)}\sum_{m=0}^{t}w(D_m(z))<\infty.
\end{align}
\end{lemma}

\begin{proof}
We first prove the necessity.
Suppose that $z\in T$ and $t\in\mathbb{Z}_+$.
If $m\in [0,t]\cap\mathbb{Z}_+$, then, for any
$x\in D_m(z)$,
$$
H^*(\mathbf{1}_{\{z\}})(x)\geq\frac{1}{ k^{m}}\geq \frac{1}{ k^{t}},
$$
and hence
$$
\bigcup_{m=0}^{t}D_m(z)\subset \left\{x\in T:\,
H^*(\mathbf{1}_{\{z\}})(x)>\frac{1}{2k^{t}}\right\}.
$$
The assumed weak-type boundedness of $H^*$, together with the disjointness of
$\{D_m(z)\}_{m\in\mathbb{Z}_+}$, yields
$$
\sum_{m=0}^{t}w(D_m(z))
\lesssim k^{pt}\|\mathbf{1}_{\{z\}}\|_{L^p(w)}^p
=k^{pt}w(z),
$$
where the implicit positive constant is independent of $z$ and $t$.
This completes the proof of \eqref{eq-Sh-Wp}, and hence the necessity.

Next, we show the sufficiency.  Let $f\in L^p(w)$, and
let $\lambda\in(0,\infty)$.  If $H^*f(x)>\lambda$ for some $x\in T$,
then there is an ancestor $z$ of $x$, with $m=d(z,x)\in[0,|x|]$, such that
$
\frac{1}{k^{m}}|f(z)|>\lambda.
$
In particular, $|f(z)|>\lambda$.  Thus, the level set
$E_\lambda$ of $H^*$ can be covered as follows:
$$
E_\lambda:=\left\{x\in T:\, H^*f(x)>\lambda\right\}
\subset\bigcup_{z\in T:\,|f(z)|>\lambda}\bigcup_{0\leq m\leq n_z}D_m(z),
$$
where $n_z$ is the largest non-negative integer satisfying
$\frac{1}{k^{n_z}}|f(z)|>\lambda$.  Applying the assumption \eqref{eq-Sh-Wp},
we obtain
\begin{align*}
w(E_\lambda)\le\sum_{z\in T:\,|f(z)|>\lambda}\sum_{m=0}^{n_z}w(D_m(z))
\lesssim\sum_{z\in T:\,|f(z)|>\lambda}k^{pn_z}w(z)
\leq \frac{1}{\lambda^{p}}\sum_{z\in T}|f(z)|^p w(z),
\end{align*}
where the last inequality follows from the definition of $n_z$.
Therefore, $H^*$ is bounded from $L^p(w)$ to $L^{p,\infty}(w)$.
This completes the proof of Lemma \ref{lem:Manc-weak}.
\end{proof}

\begin{lemma}\label{lem:weak-directional-synthesis}
Let $p\in(0,\infty)$ and let $w$ be a weight on $T$.
Suppose that $H^*$ is bounded from $L^p(w)$ to
$L^{p,\infty}(w)$ and that there exist positive constants $C$
and $\theta\in(0,1)$ such that, for any $m\in\mathbb Z_+$ and $f\in L^p(w)$,
\begin{align}\label{eq-Pm-decay-weak-synthesis}
\|P_mf\|_{L^p(w)}\leq C\theta^m\|f\|_{L^p(w)}.
\end{align}
Then $M$ is bounded from $L^p(w)$ to $L^{p,\infty}(w)$.
\end{lemma}

\begin{proof}
It follows from \eqref{eq-Ar-QP} that,
for any $f\in L^p(w)$ and $x\in T$,
\begin{align}\label{eq-Mcirc-sum-Hstar-Pm}
M^\circ f(x)\lesssim
\sum_{m\in\mathbb Z_+}H^*(P_mf)(x).
\end{align}
Let $C$ denote the implicit positive constant in
\eqref{eq-Mcirc-sum-Hstar-Pm}.
Take $\eta\in(\theta,1)$ and,
for any $\lambda\in(0,\infty)$
and $m\in\mathbb{Z}_+$, let
$$
\lambda_m:=(1-\eta)\eta^m\lambda.
$$
Using the fact that $\sum_{m\in\mathbb{Z}_+}\lambda_m=\lambda$
for any $\lambda\in(0,\infty)$,
the assumed weak-type boundedness of
$H^*$, and \eqref{eq-Pm-decay-weak-synthesis}, we conclude that,
for any $f\in L^p(w)$ and $\lambda\in (0,\infty)$,
\begin{align*}
w(\{x\in T:\,M^\circ f(x)>C\lambda\})
&\leq\sum_{m\in\mathbb Z_+}
w\left(\left\{x\in T:\,H^*(P_mf)(x)>\lambda_m\right\}\right)\\
&\lesssim\frac{1}{\lambda^{p}}\sum_{m\in\mathbb Z_+}
\frac{1}{\eta^{mp}}\|P_mf\|_{L^p(w)}^p\\
&\lesssim\frac{1}{\lambda^{p}}\|f\|_{L^p(w)}^p
\sum_{m\in\mathbb Z_+}\left(\frac{\theta}{\eta}\right)^{mp}
\lesssim\frac{1}{\lambda^{p}}\|f\|_{L^p(w)}^p.
\end{align*}
Since $M$ and $M^\circ$ are pointwise equivalent, we infer that
$M$ is bounded from $L^p(w)$ to $L^{p,\infty}(w)$. This
completes the proof of Lemma \ref{lem:weak-directional-synthesis}.
\end{proof}

We are now ready to prove Theorem \ref{thm-bound-radial}.

\begin{proof}[Proof of Theorem \ref{thm-bound-radial}]
We first prove (i). Suppose $N\in\mathbb Z_+$ and $y\in T_N$.
For any $j\in[0,N]\cap\mathbb Z_+$ and $m\in\mathbb Z_+$, let
\begin{align}\label{eq-Gammajm}
\Gamma_{j,m}(y):=\left\{x\in T:\,
d(y,x\wedge y)=j \text{ and } d(x,x\wedge y)=m\right\}.
\end{align}
Observe that the sets
$\{\Gamma_{j,m}(y)\}_{j\in[0,N]\cap\mathbb Z_+,m\in\mathbb{Z}_+}$
are mutually disjoint and
\begin{align}\label{eq-T-sum}
T=\bigcup_{j=0}^{N}\bigcup_{m\in\mathbb Z_+}\Gamma_{j,m}(y).
\end{align}
By the branch structure of $T$, we find that,
for any $j\in[0,N]\cap\mathbb Z_+$, $m\in\mathbb{Z}_+$,
and $x\in \Gamma_{j,m}(y)$, $d(x,y)=j+m$, $|x|=N-j+m$, and
\begin{align*}
\left|\Gamma_{j,m}(y)\right|=
\begin{cases}
k^m
&\text{ if }jm=0,\\
(k-1)k^{m-1}
&\text{ if }jm>0.
\end{cases}
\end{align*}
Using this and the fact that $w$ is radial, we conclude that
\begin{align}\label{eq-radial-block-weight}
w(\Gamma_{j,m}(y))=\left|\Gamma_{j,m}(y)\right|
\frac{W_{N-j+m}}{k^{N-j+m}}\sim\frac{k^j}{k^N}W_{N-j+m}.
\end{align}
We first prove (i) in the case $p\in(\frac{1}{2},1]$.
In this case, from Theorem \ref{thm-atomic-characterization},
\eqref{eq-radial-block-weight}, and \eqref{eq-T-sum},
we infer that $M$ is bounded on $L^p(w)$ if and only if
\begin{align}\label{eq-radial-double-strong}
\sup_{N\in\mathbb Z_+}\frac{1}{W_N}
\sum_{j=0}^{N}k^{(1-p)j}
\sum_{m\in\mathbb Z_+}\frac{1}{k^{pm}}W_{N-j+m}<\infty.
\end{align}
Observe that \eqref{eq-radial-double-strong} is equivalent to
\eqref{eq-radial-2} and \eqref{eq-radial-1}.
This completes the proof of (i) in the case
$p\in(\frac{1}{2},1]$.

We next prove (i) in the case $p\in(1,\infty)$.
To show its necessity, assume that $M$ is
bounded on $L^p(w)$. Note that Lemma
\ref{lem:radial-upward-necessary} yields \eqref{eq-radial-1}.
To prove the necessity of \eqref{eq-radial-2}, fix $z\in T$.
From \eqref{eq-test-M-onepoint}, we deduce that,
for any $m\in\mathbb Z_+$ and $x\in D_m(z)$,
\begin{align*}
\frac{1}{k^m}\lesssim M(\mathbf1_{\{z\}})(x).
\end{align*}
Since the sets $\{D_m(z)\}_{m\in\mathbb Z_+}$ are pairwise
disjoint, the assumed strong-type boundedness of $M$ yields
\begin{align}\label{eq-Shadow-restate}
\sum_{m\in\mathbb Z_+}\frac{1}{k^{pm}}w(D_m(z))
\lesssim\|M(\mathbf1_{\{z\}})\|_{L^p(w)}^p
\lesssim\|\mathbf1_{\{z\}}\|_{L^p(w)}^p=w(z).
\end{align}
Using the assumption that $w$ is radial, we conclude that,
for any $m,N\in\mathbb{Z}_+$ and $z\in T_N$,
\begin{align}\label{eq-radial-Dm}
w(D_m(z))=\frac{W_{N+m}}{k^N}
\text{ and }w(z)=\frac{W_N}{k^N}.
\end{align}
Thus, \eqref{eq-Shadow-restate} is precisely \eqref{eq-radial-2}.
This completes the proof of necessity of (i) in the case $p\in(1,\infty)$.
To prove its sufficiency in this case,
assume that \eqref{eq-radial-1} and \eqref{eq-radial-2} hold.
Note that Lemma \ref{lem:radial-Pm-decay} yields the
exponential decay of $\|P_m\|_{L^p(w)\to L^p(w)}$.
Moreover, by \eqref{eq-radial-Dm}, \eqref{eq-radial-2}, and
Tonelli's theorem, we find that, for any $f\in L^p(w)$,
\begin{align*}
\sum_{m\in\mathbb Z_+}\|Q_mf\|_{L^p(w)}^p
=\sum_{z\in T}|f(z)|^p\sum_{m\in\mathbb Z_+}
\frac{1}{k^{pm}}w(D_m(z))\lesssim\|f\|_{L^p(w)}^p,
\end{align*}
which, together with Lemma \ref{lem-directional-decay},
further implies the exponential decay of
$\|Q_m\|_{L^p(w)\to L^p(w)}$. From these
exponential decay estimates and \eqref{eq-Ar=PmQj},
we infer that \eqref{eq-Ar-exponential} holds.
Applying Theorem \ref{thm-Bound-Ap}, we obtain the boundedness of
$M$ on $L^p(w)$. This completes the proof of (i).

We next prove (ii). Observe that, for any
$N\in\mathbb{Z}_+$, $y\in T_N$, and $r\in\mathbb{Z}_+$,
\begin{align*}
B(y,r)=\bigcup_{j=0}^{\min\{N,r\}}\bigcup_{m=0}^{r-j}\Gamma_{j,m}(y),
\end{align*}
where $\Gamma_{j,m}(y)$ is as in \eqref{eq-Gammajm}.
Using the disjointness of $\{\Gamma_{j,m}(y)
\}_{j\in[0,N]\cap\mathbb Z_+,m\in\mathbb{Z}_+}$
and \eqref{eq-radial-block-weight}, we find that
\begin{align*}
w(B(y,r))=\sum_{j=0}^{\min\{N,r\}}\sum_{m=0}^{r-j}w(\Gamma_{j,m}(y))
\sim\frac{1}{k^N}\sum_{j=0}^{\min\{N,r\}}k^j\sum_{m=0}^{r-j}W_{N-j+m}.
\end{align*}
By Proposition \ref{prop:weak-atomic-characterization},
we find that (ii) holds.

Finally, we prove (iii). To show its necessity,
assume that $M$ is bounded from $L^p(w)$ to $L^{p,\infty}(w)$.
Observe that Lemma \ref{lem:radial-upward-necessary} guarantees \eqref{eq-radial-1}.
Since $H^*f\lesssim M^\circ f$,
Lemma \ref{lem:Manc-weak} and \eqref{eq-radial-Dm} yield
\eqref{eq-radial-3}, which completes the proof of the necessity.
To prove the sufficiency of (iii), suppose that \eqref{eq-radial-1} and
\eqref{eq-radial-3} hold. It follows from Lemma \ref{lem:Manc-weak}
that $H^*$ is bounded from $L^p(w)$ to $L^{p,\infty}(w)$.
Note that Lemma \ref{lem:radial-Pm-decay} yields
the exponential decay in \eqref{eq-Pm-decay-weak-synthesis}.
Applying Lemma \ref{lem:weak-directional-synthesis}, we obtain
the boundedness of $M$ from $L^p(w)$ to $L^{p,\infty}(w)$.
This proves (iii) and completes the proof of
Theorem \ref{thm-bound-radial}.
\end{proof}

In the next corollary, we obtain the precise ranges of $\beta\in\mathbb{R}$
such that $M$ is bounded on $L^p(w_\beta)$ and bounded
from $L^p(w_\beta)$ to $L^{p,\infty}(w_\beta)$ for any $p\in(\frac{1}{2},\infty)$,
where $w_\beta$ is as in \eqref{eq-def-radial-beta}.

\begin{corollary}\label{cor:exponential-radial}
Let $p\in(\frac{1}{2},\infty)$ and $\beta\in\mathbb R$.
The following two assertions hold.
\begin{itemize}
\item[{\rm (i)}] $M$ is bounded on
$L^p(w_\beta)$ if and only if $\beta\in(-\min\{p,1\}, p-1)$.

\item[{\rm (ii)}] $M$ is bounded from
$L^p(w_\beta)$ to $L^{p,\infty}(w_\beta)$ if and only if
either $p\in(\frac{1}{2},1)$ and $\beta\in[-p, p-1]$, or
$p\in[1,\infty)$ and $\beta\in(-1, p-1]$.
\end{itemize}
\end{corollary}

\begin{proof}
Observe that, for any $N\in\mathbb Z_+$,
\begin{align}\label{eq-WN}
W_N=w_\beta(T_N)=k^{(1+\beta)N}.
\end{align}
The main idea of this proof is to use this identity
to reduce the conditions in Theorem \ref{thm-bound-radial}
to the optimal range conditions on $\beta$. Let $c:=1+\beta$.
Note that, for any $N\in\mathbb Z_+$,
\begin{align}\label{eq-reduction-WN}
\frac{1}{W_N}\sum_{j=0}^{N}k^{(1-p)_+j}W_{N-j}
=\sum_{j=0}^{N}k^{[(1-p)_+-c]j}
\end{align}
and
\begin{align*}
\frac{1}{W_N}\sum_{m\in\mathbb Z_+}
\frac{1}{k^{pm}}W_{N+m}
=\sum_{m\in\mathbb Z_+}k^{(c-p)m}.
\end{align*}
From these two identities, we deduce that
\eqref{eq-radial-1} and \eqref{eq-radial-2}
are equivalent to $c\in((1-p)_+, p)$,
which completes the proof of (i).

We next prove (ii) by considering two
cases for the range of $p$.

\emph{Case (1)} $p\in(\frac{1}{2},1)$. In this case
we first show the necessity.
By taking $N\geq t$ in \eqref{eq-radial-weak-subone}
and by \eqref{eq-WN}, we find that, for any $t\in\mathbb{Z}_+$,
\begin{align*}
\sum_{j=0}^{t}k^{(1-c)j}
\sum_{m=0}^{t-j}k^{cm}\lesssim k^{pt}.
\end{align*}
Retaining only the terms corresponding to $j=0$ and $j=t$ and applying
the arbitrariness of $t$, we obtain $c\in[1-p, p]$.
Conversely, if $c\in[1-p, p]$, then $c\in(0,1)$.
From a geometric summation,
we infer that, for any $N, t\in\mathbb{Z}_+$,
\begin{align*}
\sum_{j=0}^{\min\{N,t\}}k^{(1-c)j}
\sum_{m=0}^{t-j}k^{cm}
&\lesssim \begin{cases}
k^{(1-c)t}&\text{ if }c\in(0,\frac{1}{2}),\\
(t+1)k^{\frac{t}{2}}&\text{ if }c=\frac{1}{2},\\
k^{ct}&\text{ if }c\in(\frac{1}{2},1)
\end{cases}\\
&\lesssim k^{pt},
\end{align*}
which, together with \eqref{eq-WN}, further implies \eqref{eq-radial-weak-subone}.
Applying Theorem \ref{thm-bound-radial}, we conclude that
$M$ is bounded from $L^p(w_\beta)$ to $L^{p,\infty}(w_\beta)$.
This completes the proof of (ii) in this case.

\emph{Case (2)} $p\in[1,\infty)$.
In this case, it follows from \eqref{eq-reduction-WN} that
\eqref{eq-radial-1} is equivalent to $c\in(0,\infty)$.
By \eqref{eq-WN}, we find that \eqref{eq-radial-3}
is equivalent to $c\in(-\infty,p]$.
This completes the proof of (ii) in this case, and hence
Corollary \ref{cor:exponential-radial}.
\end{proof}

\begin{example}\label{exam-p<1}
Let $p\in(\frac{1}{2},1]$, $\beta\in(-p,p-1)$, and
$\gamma\in\mathbb R$. For any $x\in T$, let
\begin{align*}
w_{\beta,\gamma}(x):=k^{\beta|x|}(1+|x|)^\gamma.
\end{align*}
Then $M$ is bounded on $L^p(w_{\beta,\gamma})$.
\end{example}

\begin{proof}
To prove the boundedness of $M$, by Theorem \ref{thm-bound-radial},
it suffices to verify that $w_{\beta,\gamma}$ satisfies
\eqref{eq-radial-1} and \eqref{eq-radial-2}. For this purpose, observe that,
for any $N\in\mathbb Z_+$,
\begin{align*}
W_N=w_{\beta,\gamma}(T_N)
=k^{(1+\beta)N}(N+1)^\gamma.
\end{align*}
Using the fact that $\beta\in(-p,p-1)$, we conclude that,
for any $N\in\mathbb Z_+$,
\begin{align*}
\frac{1}{W_N}\sum_{j=0}^{N}k^{(1-p)j}W_{N-j}
&=\sum_{j=0}^{N}
\frac{1}{k^{(p+\beta)j}}
\left(\frac{N-j+1}{N+1}\right)^\gamma\\
&\leq\sum_{j=0}^{N}\frac{1}{k^{(p+\beta)j}}
(1+j)^{|\gamma|}\leq\sum_{j\in\mathbb{Z}_+}
\frac{1}{k^{(p+\beta)j}}(1+j)^{|\gamma|}<\infty
\end{align*}
and
\begin{align*}
\frac{1}{W_N}\sum_{m\in\mathbb Z_+}
\frac{1}{k^{pm}}W_{N+m}&=\sum_{m\in\mathbb Z_+}
\frac{1}{k^{(p-1-\beta)m}}
\left(\frac{N+m+1}{N+1}\right)^\gamma\\
&\leq\sum_{m\in\mathbb Z_+}
\frac{1}{k^{(p-1-\beta)m}}
\left(1+m\right)^{|\gamma|}<\infty.
\end{align*}
Thus, \eqref{eq-radial-1} and
\eqref{eq-radial-2} are satisfied,
and hence the boundedness of $M$ on $L^p(w_{\beta,\gamma})$ holds.
\end{proof}

\section{Applications}\label{sec:applications}

In this section, we present several applications
of the exponential decay obtained in Theorem \ref{thm-Bound-Ap}.
To be precise, for any $p\in(\frac{1}{2},\infty)$,
under the assumption that $M$ is bounded on $L^p(w)$,
we obtain the boundedness of exponentially decaying kernel operators
on $L^p(w)$ and pointwise exponential decay of spherical
and ball averages in Subsection \ref{5.1},
and weighted Fefferman--Stein vector-valued inequalities
in Subsection \ref{5.2}.

\subsection{Exponentially Decaying Kernel Operators\label{5.1}}

In this subsection, for any $p\in(\frac{1}{2},\infty)$,
under the assumption that $M$ is bounded on $L^p(w)$, we
establish the boundedness of
exponentially decaying kernel operators on $L^p(w)$.
Moreover, we also prove pointwise exponential decay
for both spherical and ball averages.

For any $t\in\mathbb Z_+$, the operator $\mathscr R_t$ is
defined by setting, for any function $f$ on $T$
and for any $x\in T$,
\begin{align*}
\mathscr R_t f(x):=
\sum_{r\in\mathbb Z_+}(r+1)^t A_r^\circ f(x),
\end{align*}
where $A_r^\circ$ is as in \eqref{eq-def-Acirc}.
When $t=0$, we simply write $\mathscr R:=\mathscr R_0$.
The following proposition shows that
the boundedness of $M$ guarantees the boundedness of
$\mathscr R_t$ for any $t\in\mathbb Z_+$.
\begin{proposition}\label{prop:spherical-summability}
Let $p\in(\frac{1}{2},\infty)$, and let $w$ be a weight on $T$.
If $M$ is bounded on $L^p(w)$, then, for any
$t\in\mathbb Z_+$, the operator $\mathscr R_t$ is bounded
on $L^p(w)$.
\end{proposition}

\begin{proof}
By Theorem \ref{thm-Bound-Ap}, there exists
$\varepsilon\in(0,\infty)$ such that, for any
$r\in\mathbb Z_+$ and $f\in L^p(w)$,
\begin{align*}
\left\|A_r^\circ f\right\|_{L^p(w)}
\lesssim\frac{1}{k^{\varepsilon r}}\|f\|_{L^p(w)}.
\end{align*}
If $p\in[1,\infty)$, then, from Minkowski's inequality,
we infer that, for any $f\in L^p(w)$,
\begin{align*}
\left\|\mathscr R_t f\right\|_{L^p(w)}\leq
\sum_{r\in\mathbb Z_+}
(r+1)^t\left\|A_r^\circ f\right\|_{L^p(w)}
\lesssim\sum_{r\in\mathbb Z_+}
\frac{(r+1)^t}{k^{\varepsilon r}}\|f\|_{L^p(w)}
\lesssim\|f\|_{L^p(w)}.
\end{align*}
If $p\in(\frac{1}{2},1)$,
using \eqref{eq-embeding-l1} and Tonelli's theorem,
we conclude that, for any $f\in L^p(w)$,
\begin{align*}
\left\|\mathscr R_t f\right\|_{L^p(w)}^p
\leq\sum_{r\in\mathbb Z_+}(r+1)^{t p}
\left\|A_r^\circ f\right\|_{L^p(w)}^p
\lesssim\sum_{r\in\mathbb Z_+}
\frac{(r+1)^{t p}}{k^{\varepsilon pr}}
\|f\|_{L^p(w)}^p\lesssim\|f\|_{L^p(w)}^p.
\end{align*}
This completes the proof of Proposition \ref{prop:spherical-summability}.
\end{proof}

We now apply Proposition \ref{prop:spherical-summability} to
prove the boundedness of exponentially decaying kernel operators.
Suppose that $K$ is a function on $T\times T$ such that
there exist $t\in\mathbb{Z}_+$ and a positive constant $C$ such that,
for any $x,y\in T$,
\begin{align}\label{eq-decaying-kernel-condition}
|K(x,y)|\leq
C\frac{[d(x,y)+1]^t}{k^{d(x,y)}}.
\end{align}
The \emph{exponentially decaying kernel operator $\mathscr{T}_K$
associated with $K$} is defined by setting,
for any finitely supported function $f$ on $T$
and for any $x\in T$,
\begin{align*}
\mathscr{T}_Kf(x):=\sum_{y\in T}K(x,y)f(y).
\end{align*}

\begin{theorem}\label{thm:decaying-kernel}
Let $p\in(\frac{1}{2},\infty)$, and let $w$ be a weight on $T$ such
that $M$ is bounded on $L^p(w)$. Suppose that $K$ is a function
on $T\times T$ satisfying \eqref{eq-decaying-kernel-condition}.
Then $\mathscr T_K$ can uniquely extend to a bounded linear operator on $L^p(w)$.
\end{theorem}

\begin{proof}
Let $f$ be a finitely supported function on $T$.
Suppose that there exists $t\in\mathbb{Z}_+$
such that \eqref{eq-decaying-kernel-condition} holds.
For any $x\in T$, by decomposing $T$ into spheres centered
at $x$ and applying \eqref{eq-decaying-kernel-condition}, we obtain
\begin{align*}
|\mathscr T_Kf(x)|
\lesssim\sum_{r\in\mathbb Z_+}
\frac{(r+1)^t}{k^r}
\sum_{y\in S(x,r)}\left|f(y)\right|
=\sum_{r\in\mathbb Z_+}
(r+1)^t\frac{|S(x,r)|}{k^r}
A_r^\circ f(x).
\end{align*}
It follows from \eqref{eq-sphere-ball-growth} that
$\frac{|S(x,r)|}{k^r}\lesssim1$ for any $x\in T$ and $r\in\mathbb{Z}_+$.
Consequently,
$
|\mathscr T_Kf(x)|\lesssim\mathscr R_t f(x).
$
Using Proposition \ref{prop:spherical-summability}, we conclude that,
for any finitely supported function $f$ on $T$,
\begin{align*}
\left\|\mathscr T_K f\right\|_{L^p(w)}\lesssim
\left\|\mathscr R_t f\right\|_{L^p(w)}\lesssim\|f\|_{L^p(w)}.
\end{align*}
Since  the set of all finitely supported functions
is dense in $L^p(w)$, the operator $\mathscr T_K$ uniquely extends
to a bounded linear operator on $L^p(w)$.
This completes the proof of Theorem \ref{thm:decaying-kernel}.
\end{proof}

We conclude this subsection by establishing the pointwise
exponential decay for both ball and spherical averaging operators,
$A_r$ and $A_r^\circ$, defined in \eqref{eq-def-M} and \eqref{eq-def-Acirc}.

\begin{proposition}\label{prop:pointwise-decay-averages}
Let $p\in(\frac{1}{2},\infty)$, and let $w$ be a weight on $T$.
If $M$ is bounded on $L^p(w)$, then there exist
$\varepsilon\in(0,\infty)$ and a positive constant $C$
such that, for any $r\in\mathbb Z_+$,
$f\in L^p(w)$, and $x\in T$,
\begin{align}\label{eq:pointwise-decay-averages}
\left|A_r^\circ f(x)\right|+\left|A_r f(x)\right|
\le C\frac{1}{[w(x)]^{\frac{1}{p}}k^{\varepsilon r}}\|f\|_{L^p(w)},
\end{align}
where $C$ is independent of $r$, $f$, and $x$.
In particular, for any $f\in L^p(w)$ and $x\in T$,
\begin{align}\label{eq:pointwise-vanishing-averages}
\lim_{r\to\infty}A_r^\circ f(x)
=\lim_{r\to\infty}A_r f(x)=0.
\end{align}
\end{proposition}

\begin{proof}
By Theorem \ref{thm-Bound-Ap}, we find that there exists
$\varepsilon\in(0,\infty)$ such that, for any
$r\in\mathbb Z_+$ and $f\in L^p(w)$,
\begin{align*}
\left\|A_r^\circ f\right\|_{L^p(w)}\lesssim
\frac{1}{k^{\varepsilon r}}\|f\|_{L^p(w)}.
\end{align*}
Since $w(x)>0$ for any $x\in T$, it follows that,
for any $r\in\mathbb Z_+$, $f\in L^p(w)$, and $x\in T$,
\begin{align}\label{eq:spherical-pointwise-decay}
\left|A_r^\circ f(x)\right|\leq\frac{1}{[w(x)]^{\frac{1}{p}}}
\left\|A_r^\circ f\right\|_{L^p(w)}\lesssim
\frac{1}{[w(x)]^{\frac{1}{p}}k^{\varepsilon r}}\|f\|_{L^p(w)}.
\end{align}
Using the assumption $k\geq 2$,
we may assume that $\varepsilon\in(0,1)$.
Observe that, for any $x\in T$ and $r\in\mathbb{Z}_+$,
$B(x,r)$ is the disjoint union of
$\{S(x,j)\}_{j\in[0,r]\cap\mathbb{Z}_+}$.
Using this observation
and \eqref{eq-sphere-ball-growth}, we conclude that,
for any $r\in\mathbb Z_+$, $f\in L^p(w)$, and $x\in T$,
\begin{align*}
\left|A_r f(x)\right|\leq\sum_{j=0}^{r}\frac{|S(x,j)|}{|B(x,r)|}
\left|A_j^\circ f(x)\right|\lesssim\sum_{j=0}^{r}
\frac{k^j}{k^r}\left|A_j^\circ f(x)\right|.
\end{align*}
From this and \eqref{eq:spherical-pointwise-decay}, we deduce that,
for any $r\in\mathbb Z_+$, $f\in L^p(w)$, and $x\in T$,
\begin{align*}
|A_r f(x)|&\lesssim
\frac{\|f\|_{L^p(w)}}{[w(x)]^{\frac{1}{p}}}
\sum_{j=0}^{r}\frac{k^{j(1-\varepsilon)}}{k^r}
\lesssim\frac{1}{[w(x)]^{\frac{1}{p}}k^{\varepsilon r}}\|f\|_{L^p(w)}.
\end{align*}
Combining this estimate with \eqref{eq:spherical-pointwise-decay} yields
\eqref{eq:pointwise-decay-averages}. It follows immediately
from \eqref{eq:pointwise-decay-averages}
that \eqref{eq:pointwise-vanishing-averages}
holds. This completes the proof of
Proposition \ref{prop:pointwise-decay-averages}.
\end{proof}

\begin{remark}
We note that the conclusion of Proposition \ref{prop:pointwise-decay-averages}
may fail if the boundedness assumption on $M$ is removed.
Indeed, for any $x\in T$, let $w(x):=k^{-2|x|}$.
By Corollary \ref{cor:exponential-radial},
we find that $M$ is not bounded on
$L^p(w)$ for any $p\in(\frac12,\infty)$ because
the exponent $-2$ appearing in $w$
does not belong to the interval
$(-\min\{p,1\},p-1)$. Let $f=\mathbf{1}_T$.
Observe that, for any $p\in(\frac{1}{2},\infty)$,
\begin{align*}
\|f\|_{L^p(w)}^p=\sum_{N\in\mathbb{Z}_+}\frac{1}{k^{2N}}|T_N|
=\sum_{N\in\mathbb{Z}_+}\frac{1}{k^{N}}<\infty,
\end{align*}
and hence $f\in L^p(w)$. However, for any $x\in T$ and
$r\in\mathbb Z_+$, $A_r^\circ f(x)=A_rf(x)=1$.
Consequently, neither $A_r^\circ f(x)$ nor $A_rf(x)$
converges to $0$ as $r\to\infty$.
\end{remark}

\subsection{Weighted Fefferman--Stein Vector-Valued Inequalities\label{5.2}}

In this subsection, we show that, for any $p\in(\frac{1}{2},\infty)$,
the boundedness of $M$ on $L^p(w)$ is equivalent to
weighted Fefferman--Stein vector-valued inequalities for $M$.
The main result of this subsection is as follows.
\begin{theorem}\label{thm:vector-valued}
Let $p\in(\frac{1}{2},\infty)$, and let $w$ be a weight on $T$.
Then the following three statements are mutually equivalent.
\begin{itemize}
\item[{\rm (i)}] $M$ is bounded on $L^p(w)$.

\item[{\rm (ii)}] There exists a positive constant $C$ such that,
for any sequence $\{f_j\}_{j\in\mathbb{N}}$ of functions on $T$,
\begin{align}\label{eq-vector-endpoint}
\left\|\sum_{j\in\mathbb{N}}Mf_j\right\|_{L^p(w)}
\leq C\left\|\sum_{j\in\mathbb{N}}|f_j|\right\|_{L^p(w)}.
\end{align}
\item[{\rm (iii)}] There exists a positive constant $C$ such that,
for any $q\in[1,\infty]$ and for any sequence
$\{f_j\}_{j\in\mathbb{N}}$ of functions on $T$,
\begin{align}\label{eq-vector-weighted}
\left\|\left[\sum_{j\in\mathbb N}
\left(Mf_j\right)^q\right]^{\frac{1}{q}}\right\|_{L^p(w)}
\leq C\left\|\left(\sum_{j\in\mathbb N}
|f_j|^q\right)^{\frac{1}{q}}\right\|_{L^p(w)}
\end{align}
with the usual modification by the supremum when $q=\infty$.
\end{itemize}
\end{theorem}

Before proving Theorem \ref{thm:vector-valued},
we compare it with some known results.

\begin{remark}
\begin{itemize}
\item[{\rm (i)}] If $w\equiv 1$, by
Corollary \ref{cor:exponential-radial}(i) with $\beta=0$,
we find that $M$ is bounded on $L^p(T)$ if and only if
$p\in(1,\infty)$. Recall that, in \cite[Theorem 4.2]{ors21},
it was proved that, for any $1<q\leq p<\infty$ and for any
sequence $\{f_j\}_{j\in\mathbb{N}}$ of functions on $T$,
\begin{align*}
\left\|\left[\sum_{j\in\mathbb N}
\left(Mf_j\right)^q\right]^{\frac{1}{q}}\right\|_{L^p(T)}
\lesssim\left\|\left(\sum_{j\in\mathbb N}
|f_j|^q\right)^{\frac{1}{q}}\right\|_{L^p(T)}.
\end{align*}
Theorem \ref{thm:vector-valued} extends this estimate
to any weight $w$ such that $M$ is bounded on $L^p(w)$,
covers the whole range $q\in[1,\infty]$, and gives a constant independent of $q$.
The argument used in the proof of Theorem \ref{thm:vector-valued}
is different from that used in the proof of \cite[Theorem 4.2]{ors21}.
Instead of using the Fefferman--Stein inequality,
we use the fact that the boundedness of $M$ on
$L^p(w)$ guarantees the exponential decay boundedness of
spherical averaging operators $\{A^\circ_r\}_{r\in\mathbb{Z}_+}$ on $L^p(w)$
(see Theorem \ref{thm-Bound-Ap}).

\item[{\rm (ii)}] In $\mathbb{R}^n$, when $p,q\in(1,\infty)$,
the unweighted version of \eqref{eq-vector-weighted} is well known as the
Fefferman--Stein vector-valued inequality, which was
originally studied in \cite{fs71}. It is worth pointing out that,
in $\mathbb{R}^n$, the unweighted analogue of \eqref{eq-vector-weighted} fails
for $q=1$ (see \cite[Chapter II, Section 5.1]{ste93}
for a concrete counterexample).
\end{itemize}
\end{remark}

We now prove Theorem \ref{thm:vector-valued}.
\begin{proof}
The implication (iii) $\Longrightarrow$ (ii) is obvious; we omit the details.
Applying  \eqref{eq-vector-endpoint} to a sequence containing only one
non-zero function, we obtain the implication (ii) $\Longrightarrow$ (i).
It remains to prove the implication (i) $\Longrightarrow$ (iii).
Suppose that $M$ is bounded on $L^p(w)$.
By Proposition \ref{prop:spherical-summability}, we find that
the operator $\mathscr{R}$ is also bounded on $L^p(w)$.
We now show \eqref{eq-vector-weighted} by considering two
cases for the range of $q$.

\emph{Case (1)} $q\in[1,\infty)$. In this case,
using the pointwise equivalence
between $M$ and $M^\circ$ and using Minkowski's inequality,
we conclude that, for any sequence $\{f_j\}_{j\in\mathbb{N}}$
of functions on $T$ and for any $x\in T$,
\begin{align*}
\left(\sum_{j\in\mathbb N}
\left[Mf_j(x)\right]^q\right)^{\frac{1}{q}}
\lesssim\left\{\sum_{j\in\mathbb N}
\left[\sum_{r\in\mathbb Z_+}A_r^\circ f_j(x)
\right]^q\right\}^{\frac{1}{q}}
\leq\sum_{r\in\mathbb Z_+}\left\{
\sum_{j\in\mathbb N}\left[A_r^\circ f_j(x)\right]^q
\right\}^{\frac{1}{q}}.
\end{align*}
For any fixed $r\in\mathbb{Z}_+$, Minkowski's inequality
for the average over $S(x,r)$ yields
$$
\left\{\sum_{j\in\mathbb N}\left[A_r^\circ f_j(x)\right]^q
\right\}^{\frac{1}{q}}\leq A_r^\circ
\left(\left(\sum_{j\in\mathbb N}
|f_j|^q\right)^{\frac{1}{q}}\right)(x).
$$
Consequently,
\begin{align*}
\left(\sum_{j\in\mathbb N}
\left[Mf_j(x)\right]^q\right)^{\frac{1}{q}}
\lesssim\sum_{r\in\mathbb Z_+} A_r^\circ
\left(\left(\sum_{j\in\mathbb N}
|f_j|^q\right)^{\frac{1}{q}}\right)(x)
=\mathscr{R}\left(\left(\sum_{j\in\mathbb N}
|f_j|^q\right)^{\frac{1}{q}}\right)(x).
\end{align*}
Taking the $L^p(w)$ norm and applying the
boundedness of $\mathscr{R}$, we further obtain
$$
\left\|\left[\sum_{j\in\mathbb N}
\left(Mf_j\right)^q\right]^{\frac{1}{q}}\right\|_{L^p(w)}
\lesssim\left\|\left(\sum_{j\in\mathbb N}
|f_j|^q\right)^{\frac{1}{q}}\right\|_{L^p(w)}.
$$
Note that all the implicit positive constants are independent of $q$.
Thus, \eqref{eq-vector-weighted} holds in this case.

\emph{Case (2)} $q=\infty$. In this case,
for any sequence $\{f_j\}_{j\in\mathbb{N}}$
of functions on $T$ and for any $x\in T$,
\begin{align*}
\sup_{j\in\mathbb N}Mf_j(x)&\lesssim
\sup_{j\in\mathbb N}\sum_{r\in\mathbb Z_+}A_r^\circ f_j(x)
\leq\sum_{r\in\mathbb Z_+}
\sup_{j\in\mathbb N}A_r^\circ f_j(x)\\
&\leq\sum_{r\in\mathbb Z_+}
A_r^\circ\left(\sup_{j\in\mathbb N}\left|f_j\right|\right)(x)
=\mathscr{R}\left(\sup_{j\in\mathbb N}\left|f_j\right|\right)(x).
\end{align*}
From the boundedness of $\mathscr{R}$, we deduce that
\eqref{eq-vector-weighted} holds in this case.
This completes the proof of the implication
(i) $\Longrightarrow$ (iii), and hence Theorem \ref{thm:vector-valued}.
\end{proof}

%%%%%%%%%%%%%%%%%%%%%%%%%%%%%%%%%%%%%%%

\smallskip
\noindent\textbf{Acknowledgements}\quad
The authors acknowledge the use of AI tools during the exploratory stage of this project.
All mathematical arguments and proofs in the final manuscript were checked
and written by the authors.

\smallskip
\noindent\textbf{Data availability}\quad
Data sharing is not applicable to this article as no data sets were generated or analysed.

\section*{Declarations}

\noindent\textbf{Conflict of interest}\quad
The authors state that there is no conflict of interest.

\bigskip

\noindent Dachun Yang, Wen Yuan and Mingdong Zhang.

\medskip

\noindent Laboratory of Mathematics and Complex Systems
(Ministry of Education of China),
School of Mathematical Sciences,
Institute for Advanced Study, Beijing Normal University,
Beijing 100875, The People's Republic of China

\smallskip

\noindent{{\it E-mails:}}
\texttt{dcyang@bnu.edu.cn} (D. Yang)

\noindent\phantom{{\it E-mails:}}
\texttt{wenyuan@bnu.edu.cn} (W. Yuan)

\noindent\phantom{\it E-mails:}
\texttt{mdzhang@mail.bnu.edu.cn} (M. Zhang)
\end{document}